\documentclass[12pt]{amsart}   	
\usepackage{geometry}                		
\usepackage{graphicx}				
\usepackage{amssymb}
\usepackage{mathtools}

\usepackage{amsthm,amssymb,amsmath,amsfonts,extarrows}
\usepackage[all,cmtip,2cell]{xy}
\usepackage{enumitem}
\usepackage[colorlinks=true]{hyperref}

\usepackage[dvipsnames]{xcolor}

\usepackage{comment}
\usepackage{tikz-cd}

\newtheorem{theorem}{Theorem}[section]
\newtheorem{lemma}[theorem]{Lemma}
\newtheorem{cor}[theorem]{Corollary}

\newtheorem{prop}[theorem]{Proposition}
\theoremstyle{definition}

\newtheorem{remark}[theorem]{Remark}

\newtheorem*{theorem*}{Theorem}

\numberwithin{equation}{section}

\newcommand{\overbar}[1]{\mkern 1.5mu\overline{\mkern-1.5mu#1\mkern-1.5mu}\mkern 1.5mu}
\newcommand{\CC}{\mathbb{C}}

\newcommand{\NN}{\mathbb{N}}

\newcommand{\PP}{\mathbb{P}}
\newcommand{\QQ}{\mathbb{Q}}
\newcommand{\RR}{\mathbb{R}}

\newcommand{\ZZ}{\mathbb{Z}}
\newcommand{\bA}{\mathbf{A}}

\newcommand{\bH}{\mathbf{H}}
\newcommand{\bM}{\mathbf{M}}

\newcommand{\bx}{\mathbf{x}}
\newcommand{\by}{\mathbf{y}}
\newcommand{\bz}{\mathbf{z}}

\newcommand{\calA}{\mathcal{A}}
\newcommand{\calB}{\mathcal{B}}
\newcommand{\calC}{\mathcal{C}}
\newcommand{\calD}{\mathcal{D}}
\newcommand{\calE}{\mathcal{E}}

\newcommand{\calJ}{\mathcal{J}}
\newcommand{\calL}{\mathcal{L}}

\newcommand{\calO}{\mathcal{O}}

\newcommand{\calU}{\mathcal{U}}
\newcommand{\calX}{\mathcal{X}}
\newcommand{\calY}{\mathcal{Y}}

\newcommand{\gothA}{\mathfrak{A}}
\newcommand{\gothC}{\mathfrak{C}}

\newcommand{\hfal}{h_{\mathrm{Fal}}}

\DeclareMathOperator{\Hilb}{Hilb}
\DeclareMathOperator{\image}{Im}
\DeclareMathOperator{\id}{id}

\DeclareMathOperator{\Jac}{Jac}

\DeclareMathOperator{\rank}{rank}

\DeclareMathOperator{\Spec}{Spec}

\DeclareMathOperator{\Stab}{Stab}

\begin{document}
\title{Uniformity of quadratic points}
\author{Tangli Ge}

\address{Department of Mathematics\\
Brown University\\
Box 1917\\
Providence, RI 02912\\
U.S.A.}
\email{\url{tangli@princeton.edu}}

\thanks{Work supported in part by funds from NSF grant
   DMS-2100548 and DMS-1759514}

\begin{abstract} 
In this paper, we extend a uniformity result of Dimitrov--Gao--Habegger \cite{DGH21} to dimension two and use it to get a uniform bound on the cardinality of the set of all quadratic points for non-hyperelliptic non-bielliptic curves which only depends on the Mordell--Weil rank, the genus of the curve and the degree of the number field.
\end{abstract} 
\maketitle

\section{Introduction}

Let $C$ be a smooth, geometrically irreducible, projective curve defined over a number field $F$, of genus at least $2$. The Mordell conjecture proved by Faltings \cite{Faltings83} states that the set of rational points $C(F)$ on $C$ is finite. If $C$ is moreover non-hyperelliptic and non-bielliptic (i.e., having no $2$-to-$1$ map to an elliptic curve), Harris--Silverman showed in \cite[Corollary 3]{SilvermanHarris} that the set of quadratic points on $C$ is indeed also finite, by applying Faltings's Theorem \cite{Faltings} to its symmetric product. Motivated by this and the recent uniformity result of Dimitrov--Gao--Habegger \cite{DGH21}, we will show that
  
\begin{theorem}\label{quadratic_points}
Let $g\geq 3$ and $d\geq 1$ be integers. Then there exists a constant $c=c(g,d)>0$ such that for any non-hyperelliptic, non-bielliptic, smooth, geometrically irreducible, projective curve $C$ of genus $g$ defined over a number field $F$ with $[F:\QQ]\leq d$, we have
\[
\# C(F,2)\leq c^{1+\rho}
\]where $C(F,2)$ is the set of points on $C$ that are defined over some $F'$ with $[F':F]\leq 2$ and $\rho$ is the rank of the Mordell--Weil group $\Jac(C)(F)$.
\end{theorem}

Let us call a positive dimensional variety $X$ defined over a field $K$ \emph{geometrically Mordellic} or \emph{GeM}, if the base change  $X_{\bar K}$ to an algebraic closure $\bar K$ of $K$ does not contain subvarieties which are not of general type (see \cite[Definition 1.4]{DV}). For a subvariety of an abelian variety, Ueno  \cite{Ueno} showed that being GeM is equivalent to containing no translates of positive dimensional abelian subvarieties.

Let $\bA_g$ be the fine moduli space of principally polarized abelian varieties over $\bar\QQ$ with symplectic level-$l$-structure where we fix $l\geq 3$ and a primitive $l$-th root of unity; see \S\ref{section_notation_abelianschemes} for more details. Fix a height function $h:\bA_g(\bar\QQ)\rightarrow \RR$ associated to some ample line bundle on a fixed projective compactification $\overbar{\bA_g}$ of $\bA_g$ (see \S\ref{section_notation} for more details). The above theorem relies on a uniform bound of the same quality as in \cite{DGH21}, for the number of rational points on the fibers of a family of $2$-dimensional GeM subvarieties of abelian varieties. Specifically we show the following

\begin{theorem}\label{mainthm}
		Let $S$ be a variety over $\bar\QQ$ and let $\calA\rightarrow S$ be a principally polarized abelian scheme with symplectic level-$l$-structure. Let $\calX\subseteq\calA$ be a closed subvariety such that  each fiber $\calX_s$ is a $2$-dimensional integral GeM subvariety in $\calA_s$, with $\calX_s$ generating $\calA_s$ for any $s\in S(\bar\QQ)$ (see the end of \S\ref{section_notation_abelianschemes}).  Assume that the modular map $\iota_S:S\rightarrow\bA_g$ induced by the family $\calA\rightarrow S$ is quasi-finite. Then there is some $c>0$, depending on all the information above, with the following property: for any $s\in S(\bar\QQ)$ satisfying $h(\iota_S(s))>c$, and any subgroup $\Gamma\subseteq \calA_s(\bar\QQ)$ of finite rank $r$, we have
	\[
	\#(\calX_s(\bar\QQ)\cap\Gamma)\leq c^{1+r}.
	\]
\end{theorem}

The idea of our proof is as follows. We first use the non-degeneracy result by Gao \cite[Theorem 10.1]{Gao_Bettirank} and the height inequality developed by Dimitrov-Gao-Habegger \cite[Theorem 1.6]{DGH21} to study the distance between the algebraic points in $\calX(\bar\QQ)$ using N\'eron--Tate height. We find several families of curves in $\calX\rightarrow S$, over possibly different bases, together with one marked family over $S$, which satisfy the following: for any point $P$ of $\calX(\bar\QQ)$ not on the marked family, all of the points that are ``close''  to $P$ live on one single curve among these families of curves (see Lemma \ref{lemma_surface}). Then we prove and use the following Theorem \ref{curve} to bound the number of algebraic points in $\Gamma$ on those families of curves. We derive that for any rational point away from the marked family, there is a uniform bound on the number of rational points that are close to it.  Finally we combine this with a classical result of R\'emond for large points to get the desired bound on $\#(\calX_s(\bar\QQ)\cap\Gamma)$ for any $s\in S(\bar\QQ)$ of large moduli height.

\begin{theorem}\label{curve}
	Let $S$ be a variety over $\bar\QQ$ and let $\calA\rightarrow S$ be an abelian scheme. Let $\calC\subseteq \calA$ be a closed subvariety whose fiber $\calC_s$ over $S$ is a GeM curve in $\calA_s$, for any $s\in S(\bar\QQ)$. Then there exists $c>0$ which only depends on the family $\calC\subseteq\calA\rightarrow S$, such that for any $s\in S(\bar\QQ)$ and any subgroup $\Gamma\subseteq\calA_s(\bar\QQ)$ of finite rank $r$, 
	\[
	\#(\calC_s(\bar\QQ)\cap \Gamma) \leq c^{1+r}.
	\]
\end{theorem}

Theorem \ref{curve} itself is of interest to us. Compared with \cite[Theorem 1.1]{Gao_survey} or \cite[Theorem 4]{Kuhne}, it is also valid for the case when the ambient abelian varieties are not the Jacobians of the curves. Also, the curves can be singular or reducible. 

The main new ingredient in the proof of Theorem \ref{curve} is the use of Hilbert schemes to construct non-degenerate subvarieties of a family of abelian varieties which could, for example, even be a constant family of abelian varieties (see \S\ref{section_nondegeneracy}). This is motivated by the work of Hrushovski \cite[Lemma 1.3.2]{Hrushovski} and Scanlon \cite{Scanlon} on automatic uniformity, and the Betti rank formula by Gao \cite[Theorem 10.1]{Gao_Bettirank}.


In a joint project \cite{GGK} with Gao and K\"uhne, we prove the uniform Mordell-Lang conjecture for any subvariety of an abelian variety. We generalize the uniform Bogomolov conjecture there and compare the distance of algebraic points with the moduli height of the abelian variety. The non-degeneracy construction using Hilbert schemes in \S\ref{section_nondegeneracy} is used. Compared to the current paper, some finer properties of Hilbert schemes are also required. As a particular consequence, the constant $c$ in Theorem \ref{quadratic_points} can be made independent of the degree $[F:\QQ]$.

\section{Notations}\label{section_notation}
We fix conventions and notations in this section.
\subsection{General varieties}
The algebraic closure of $\QQ$ in $\CC$ is denoted by $\bar\QQ$. We work exclusively over $\bar\QQ$ except in the proof of Theorem \ref{quadratic_points} and in the Appendix \ref{Appendix B}. Varieties are separated schemes \emph{of finite type} over $\bar\QQ$; in particular, varieties are Noetherian. Integral varieties are irreducible and reduced varieties. All subvarieties are required to be Zariski closed, unless otherwise stated. An algebraic point $x$ of a variety $X$ is a morphism $x:\Spec\bar\QQ\rightarrow X$; the set of algebraic points of $X$ is denoted by $X(\bar\QQ)$. Curves are $1$-dimensional varieties. Surfaces are $2$-dimensional varieties. In particular, curves and surfaces can have lower dimensional irreducible components.

A family of varieties is a dominant algebraic family $\calU\rightarrow S$, where $S$ is the base variety and $\calU$ is the total space which is also a variety; the word ``family'' is used here to stress that we are interested in studying the fibers. We do not require a family to be flat, although we will reduce to the flat case in the proofs. A family of curves (resp. surfaces) is usually written as $\calC\rightarrow S$ (resp. $\calX\rightarrow S$). Denote by $\calU_S^n$ the $n$-th fibered power of $\calU$ over $S$; since our fiber products for a family are always taken over the base of the family, we often leave out the subscript and write $\calU^n$ for simplicity. If $\calU\rightarrow S$ and $T\rightarrow S$ are two varieties over $S$, denote the base change $\calU\times_S T$ by $\calU_T$.  Denote the fiber of $\calU\rightarrow S$ over a point $s\in S$ by $\calU_s$.

A morphism of varieties is called quasi-finite, if every geometric fiber is finite. A morphism of varieties $f: X\rightarrow Y$ is called \emph{generically finite} (to its image) if there is a dense open subset $U$ of $X$ such that $f|_U: U\rightarrow Y$ is quasi-finite; see \cite[\href{https://stacks.math.columbia.edu/tag/073A}{Remark 073A}]{stacks-project}. In other words, we define the generically finiteness as a property local on the source. Clearly, a morphism of varieties $f:X\rightarrow Y$ with $X$ irreducible is generically finite if and only if $\dim f(X)=\dim X$. Remark also that the composition of two generically finite morphisms is not necessarily generically finite under this definition. 

A morphism of varieties $\calX\rightarrow S$ is called (resp. quasi-)projective if it factors through a (resp. locally) closed immersion $\calX\hookrightarrow \PP(\calE)$ for some coherent sheaf $\calE$ on $S$. We usually assume $S$ is quasi-projective over $\bar\QQ$; then we can indeed have a stronger projectivity with a closed immersion $\calX\hookrightarrow\PP^N_S$ for some $N$; see \cite[Summary 13.71]{GW} for a discussion of various notions of projectivity.

Let $S$ be a variety and let $\calX\rightarrow S$ be a projective morphism of varieties. There exists a fine moduli space, namely the (relative) Hilbert scheme, denoted by $\Hilb(\calX/S)$, which represents the functor that associates to any locally noetherian scheme $T$ over $S$ the set of all closed subschemes of $\calX_T$ that are flat over $T$. The Hilbert scheme $\Hilb(\calX/S)$ has a stratification by Hilbert polynomials with respect to a given relatively very ample line bundle on $\calX$ as
\[\Hilb(\calX/S)=\bigsqcup_{\Phi\in \QQ[\lambda]}\Hilb^{\Phi}(\calX/S)
\]such that each piece $\Hilb^\Phi(\calX/S)$ is a projective scheme over $S$;  cf. \cite{Gro_Hilbert} and \cite{Nitsure_hilbertscheme}.

\subsection{Abelian schemes}\label{section_notation_abelianschemes}
Let $S$ be a variety and let $\calA\rightarrow S$ be an abelian scheme of (relative) dimension $g$. By \cite[1.10(a)]{Faltings-Chai}, the abelian scheme $\calA\rightarrow S$ is projective if $S$ is a Noetherian normal scheme. For example, if $S$ is a smooth variety over $\bar\QQ$, then $\calA\rightarrow S$ is projective. 
Fix some $l\geq 3$ throughout. By level-$l$-structure on $\calA\rightarrow S$, we mean \emph{symplectic} level-$l$-structure; in characteristic zero, it is an isomorphism of $S$-group schemes $(\ZZ/l\ZZ)^{2g}\xrightarrow{\sim}\calA[l]$ which takes the standard symplectic pairing to the Weil pairing. We fix a choice of primitive $l$-th root of unity throughout.

The moduli space of principally polarized abelian varieties over $\bar\QQ$ of dimension $g$ with level-$l$-structure is representable by a quasi-projective smooth variety denoted by $\bA_g$; see \cite[Theorem 2.3.1]{Genestier-Ngo} and \cite[Theorem 7.9 and its proof]{GIT} for its existence and properties. Let $\gothA_g\rightarrow \bA_g$ be the universal abelian scheme, which is projective by the last paragraph due to the smoothness of $\bA_g$. So there exist relatively ample line bundles for $\gothA_g\rightarrow \bA_g$.

If $\calA\rightarrow S$ is principally polarized with level-$l$-structure, then there exist natural modular maps $\iota_S:S\rightarrow \bA_g$ and $\iota:\calA\rightarrow \gothA_g$ such that $\calA\rightarrow S$ is the pullback of $\gothA_g\rightarrow \bA_g$ by $\iota_S$. In this case, the abelian scheme $\calA\rightarrow S$ is projective.

Fix a relatively ample, symmetric (i.e. $[-1]^*\calL_0\cong\calL_0$) line bundle $\calL_0$ for the universal family $\gothA_g\rightarrow\bA_g$ which defines fiber-wise N\'eron--Tate heights $\hat h:\gothA_g(\bar\QQ)\rightarrow\RR_{\geq 0}$ (e.g., given a relatively ample line bundle $\calL$, we can take  $\calL_0:=\calL\otimes [-1]^*\calL$ which is relatively ample and symmetric). Fix an ample line bundle $M_0$ on a fixed projective compactification $\bar\bA_g$ of $\bA_g$, which defines a height function $h:\bA_g(\bar\QQ)\rightarrow \RR$ up to a bounded function. 

By the $n$-th Faltings--Zhang morphism $\calD_n$ for $\calA\rightarrow S$, we mean the following morphism 
\[
\begin{split}
	\calD_n:\quad\calA^{n+1}&\longrightarrow\calA^n\\
	(P_0,\dots,P_n)&\longmapsto (P_1-P_0,\dots,P_n-P_0).
\end{split}
\]

We say that an irreducible subvariety $X$ of an abelian variety $A$ generates $A$, if $X-X$ is not contained in any proper abelian subvariety of $A$. In general, we say that a (possibly reducible) subvariety $X\subseteq A$ generates $A$ if some irreducible component of $X$ generates $A$.

\section{Non-degeneracy}\label{section_nondegeneracy}
Let us start by defining the Betti maps. Let $S$ be an integral quasi-projective variety (over $\bar\QQ$ as always).  Consider an abelian scheme $\pi:\calA\rightarrow S$ and a smooth complex point $s\in S(\CC)$. One can take a local trivialization which is called a \emph{Betti map} $b_{\Delta}:\calA_\Delta\rightarrow \mathbb{T}^{2g}$  over any simply-connected analytic open neighborhood $\Delta\subseteq S^{\text{sm,an}}$ of $s$ using the period mappings, where $S^{\text{sm,an}}$ denotes the analytification of the smooth locus of $S$, $\calA_\Delta:=\pi^{-1}(\Delta)$ denotes the corresponding complex abelian scheme over $\Delta$ and $\mathbb{T}^{2g}$ is the real torus of dimension $2g$. The Betti map $b_\Delta$ is a real analytic map of analytic manifolds with complex analytic fibers. 

For an integral subvariety $\calX$ of $\calA$ dominant over $S$, restrict a Betti map to $\calX^{\text{sm,an}}\cap \calA_\Delta$. The \emph{generic Betti rank} is defined to be the maximal $\RR$-rank of $(\mathrm{d}b_\Delta|_{\calX^{\text{sm,an}}\cap \calA_\Delta})_x$ for $x\in {\calX^{\text{sm,an}}\cap \calA_\Delta}$, where $\mathrm{d}f$ is the differential for any smooth morphism $f$ between manifolds. This definition is independent of the choice of the Betti map. See  \cite[\S4]{Gao_Bettirank} and \cite[\S B.1]{DGH21} for more details of the construction and the properties.

The notion of non-degeneracy first appeared in the work of Habegger  \cite{Habegger}. Here we take the definition from \cite[Definition B.4]{DGH21}. We say an irreducible dominant-over-$S$ closed subvariety $\calX$
 is \emph{non-degenerate} if the generic Betti rank of $\calX^{\text{red}}$ is equal to $2\dim \calX$, where $\dim\calX$ is the dimension of $\calX$ as a scheme.

 In this section, we will use Gao's Betti rank formula \cite[Theorem 10.1]{Gao_Bettirank} to construct a new type of non-degenerate subvarieties. 

Let us start with the following lemma:
\begin{lemma}\label{lemma_hilb}
	Let $\calY\rightarrow B$ be a projective morphism of varieties over $\bar\QQ$. Let $H\subseteq \Hilb^\Phi(\calY/B)$ be a locally closed integral subvariety (for some Hilbert polynomial $\Phi\in \QQ[\lambda]$ with respect to some relatively very ample line bundle). Let 
	\[
	\begin{tikzcd}
		\calU\arrow[r,hook]\arrow[d] &\calY_H=\calY\times_B H\arrow[ld] \\
		H & 
	\end{tikzcd}
	\] be the flat family induced by $H\subseteq \Hilb^\Phi(\calY/B)$. Assume that its geometric generic fiber $\calU_{\bar\eta}$ is integral.
	
	Then the fibered power $\calU^n_H$ is irreducible, and for $n\gg1$, the composition map $f_n:\calU^n_H\hookrightarrow (\calY_H)^n_H\rightarrow \calY^n_B$ is generically finite. 
\end{lemma}
\begin{proof}
Since the geometric generic fiber $\calU_{\bar\eta}$ is integral, there is a dense open subset $V\subseteq H$ such that $\calU_V\rightarrow V$ has geometrically integral fibers; see \cite[\href{https://stacks.math.columbia.edu/tag/0559}{Lemma 0559}]{stacks-project} and \cite[\href{https://stacks.math.columbia.edu/tag/0578}{Lemma 0578}]{stacks-project}.

For the first part, since the morphism $\calU^n_H\rightarrow H$ is open due to flatness, with irreducible fibers over the dense subset $V$, the irreducibility of $\calU^n_H$ follows from a topological argument as in \cite[\href{https://stacks.math.columbia.edu/tag/004Z}{Lemma 004Z}]{stacks-project}.

For the second part, by definition of generically finiteness, it suffices to replace $H$ by $V$ and thus we assume without loss of generality that $\calU\rightarrow H$ has geometrically integral fibers. 

\emph{Claim: if $n$ is large, there is a closed point in the image of $f_n$ whose fiber is a singleton.}

To prove the claim, without loss of generality, assume that $B=\Spec \bar\QQ$ is a point. Let $h_0\in H(\bar\QQ)$ and denote the integral fiber over $h_0$ by $\calU_0$. Since the fibers of $\calU\rightarrow H$ are integral and have the same Hilbert polynomial and distinct moduli, we have
\begin{equation}\label{eqn_key}
	\bigcap_{P\in \calU_0(\bar\QQ)}\{h\in H(\bar\QQ):P\in \calU_h(\bar\QQ)\}=\{h_0\}.
\end{equation}

Observe that the left hand side is an intersection of Zariski closed subsets of $H$. It is indeed a finite intersection by Noetherian property. For $n\gg1$, there exist $P_1,P_2,\dots,P_n\in \calU_0(\bar\QQ)$ such that 
\[
\bigcap_{i=1}^n\{h\in H(\bar\QQ):P_i\in \calU_h(\bar\QQ)\}=\{h_0\}.
\]Then the fiber of $f_n$ over $(P_1,\dots,P_n)\in \calY^n_B(\bar\QQ)$ is a singleton. Thus the claim follows.

The generically finiteness is an immediate result of the following Chevalley's upper semicontinuity theorem on the dimension of fibers:

\begin{theorem}[{\cite[Th\'eor\`eme 13.1.3]{EGAIV}}]
	Suppose $f:X\rightarrow Y$ is a morphism of schemes that is locally of finite type. Then the function $X\rightarrow \NN$ which sends $x\in X$ to $\dim_x X_{f(x)}$, that is, the local dimension of the fiber over $f(x)$ at $x$, is upper semicontinuous.
\end{theorem}
In fact, Chevalley's theorem together with the claim implies that there is a dense open subset $W\subseteq \calU^n_H$, such that for the restriction $f_n|_W: W\rightarrow \calY^n_B$, any point $x\in W$ is isolated in the fiber $W_{f_n(x)}$. In other words, all the (geometric) fibers of $f_n|_W$ are finite. 
\end{proof}

\begin{remark}
	In the proof, the equation \eqref{eqn_key} holds true if we only assume that the fiber $\calU_0$ is reduced and the fibers have distinct moduli. However, as pointed out by the referee, we need the irreducibility of the geometric generic fiber as a sufficient condition to ensure the irreducibility of the fibered powers.
\end{remark}

The next corollary allows us to have a little more flexibility:
\begin{cor}\label{generically_finite_is_ok}
Let $\calY\rightarrow B$ be a projective morphism of varieties over $\bar\QQ$.  Let $H$ be an irreducible variety together with a generically finite morphism $H\rightarrow \Hilb(\calY/B)$. Let 
	\[
	\begin{tikzcd}
		\calU\arrow[r,hook]\arrow[d] &\calY_H=\calY\times_B H\arrow[ld] \\
		H & 
	\end{tikzcd}
	\] be the flat family induced by $H\rightarrow\Hilb(\calY/B)$. Assume that its geometric generic fiber $\calU_{\bar\eta}$ is integral. 
	
	Then the fibered power $\calU^n_H$ is irreducible, and for $n\gg1$, the composition map $f_n:\calU^n_H\hookrightarrow (\calY_H)^n_H\rightarrow \calY^n_B$ is generically finite.  
\end{cor}
\begin{proof}
The first part is the same as in the proof of Lemma \ref{lemma_hilb}. Note that the irreducibility of $H$ ensures that $H$ actually maps into $\Hilb^\Phi(\calY/B)$ for some Hilbert polynomial $\Phi$. We may assume without loss of generality that $H$ is reduced.

For the second part, let $H'$ be the scheme-theoretic image of $H$ in $\Hilb(\calY/B)$, which is the Zariski closure of the image of $H$ equipped with the reduced structure. Let $\calU'\rightarrow H'$ be the corresponding flat family. Then $\calU\rightarrow H$ is a base change of $\calU'\rightarrow H'$. Since $H\rightarrow H'$ is generically finite and dominant, the geometric generic fiber of $\calU'\rightarrow H'$ is just $\calU_{\bar\eta}$, which is integral. By Lemma \ref{lemma_hilb}, the natural map $(\calU')^n_{H'}\rightarrow \calY^n_B$ is generically finite. Then $f_n$ as the composition of $\calU^n_H\rightarrow (\calU')^n_{H'}$, which is a base change of $H\rightarrow H'$, whence generically finite, with $(\calU')^n_{H'}\rightarrow \calY^n_B$, is generically finite.
\end{proof}

\begin{prop}\label{nondegeneracy}
	Let $\gothA_g\rightarrow \bA_g$ be the universal abelian scheme defined in \S\ref{section_notation_abelianschemes}. Let $S$ be an integral quasi-projective variety over $\bar\QQ$ with a generically finite morphism $S\rightarrow \Hilb(\gothA_g/\bA_g)$, and let $\calU\rightarrow S$ be the induced family inside $\calA:=\gothA_g\times_{\bA_g}S\rightarrow S$. For the geometric generic fiber $\calU_{\bar\eta}$, assume the following:
	\begin{enumerate}
		\item $\calU_{\bar\eta}$ is integral,
		\item $\calU_{\bar\eta}$ generates $\calA_{\bar\eta}$, and
		\item $\calU_{\bar\eta}$ has finite stabilizer in $\calA_{\bar\eta}$.
	\end{enumerate}
	 Then for $n\gg1$, the fibered power $\calU^n\subseteq \calA^n$ is irreducible and non-degenerate.
\end{prop}
\begin{proof}
	Note that $\calU\rightarrow S$ satisfies the assumptions (e.g., $\dim\calU>\dim S$ follows from (2) and $g>0$) of Gao's Betti rank formula; see \cite[Theorem 10.1(i)]{Gao_Bettirank} and the Erratum \cite{GaoErratum}. So for each $t\geq0$, 
		 the generic Betti rank of $\calU^n$ is at least $2(\dim \iota(\calU^n)-t)$ for all $n\geq \dim S-t$, where $\iota:\calU^n\rightarrow \gothA_{gn}$ is the modular map. But by Corollary \ref{generically_finite_is_ok}, if $n\geq N$ for some $N\gg1$, the map $\calU^n\rightarrow \gothA_{g}^n$ is generically finite. Since the natural map $\gothA_{g}^n\rightarrow\gothA_{gn}$ is clearly quasi-finite, it implies $\iota:\calU^n\rightarrow \gothA_g^n\rightarrow\gothA_{gn}$ is generically finite. In particular, $\dim\iota(\calU^n)=\dim \calU^n$. Therefore, taking $t=0$ and $n\geq \max\{N,\dim S\}$, we deduce that the generic Betti rank of $\calU^n$ is equal to $2\dim \calU^n$. In other words, $\calU^n\subseteq \calA^n$ is non-degenerate. 
\end{proof}

\begin{remark}
	If the modular map $S\rightarrow \bA_g$ for a principally polarized abelian scheme $\calA\rightarrow S$ with level-$l$-structure is generically finite, then for any flat algebraic family $\calU\rightarrow S$ in $\calA\rightarrow S$, the induced map $S\rightarrow \Hilb(\gothA_g/\bA_g)$ is generically finite.  
\end{remark} 

Sometimes we would also like to consider the non-degeneracy of the image of the Faltings--Zhang morphism.  Let $S$ be a variety over $\bar\QQ$ with a morphism $S\rightarrow \Hilb(\gothA_g/\bA_g)$, and let $f:\calU\rightarrow S$ be the induced family inside the abelian scheme $\calA:=\gothA_g\times_{\bA_g}S$. Consider the following embedding of $\calU\times_S \calU$ in $\calU\times_S \calA$:
\[
	\begin{split}
		\calU\times_S \calU&\longrightarrow \calU\times_S \calA\\
		(P,Q)&\longmapsto (P,Q-P).
	\end{split}
	\]
	Denote its isomorphic image by $\calU_1$ (the inverse map is given by $(P,x)\mapsto(P,P+x)$). It is equipped with a natural flat morphism to $\calU$ through the first projection $\calU_1\subseteq\calU\times_S\calA\rightarrow \calU$, whose fiber over any point $P\in \calU(\bar\QQ)$ is $\calU_{f(P)}-P$. We have the following non-degeneracy result in terms of the Faltings--Zhang morphism:

\begin{prop}\label{nondegeneracy_for_FZ}
	Let $\gothA_g\rightarrow \bA_g$ be the universal abelian scheme defined in \S\ref{section_notation_abelianschemes}. Let $S$ be an integral quasi-projective variety over $\bar\QQ$ with a morphism $S\rightarrow\Hilb(\gothA_g/\bA_g)$, and let $\calU\rightarrow S$ be the induced family inside the abelian scheme $\calA:=\gothA_g\times_{\bA_g}S\rightarrow S$. For the geometric generic fiber $\calU_{\bar\eta}$, assume the following:
	\begin{enumerate}
		\item $\calU_{\bar\eta}$ is integral,
		\item $\calU_{\bar\eta}$ generates $\calA_{\bar\eta}$, and
		\item $\calU_{\bar\eta}$ has finite stabilizer in $\calA_{\bar\eta}$.
	\end{enumerate}
	Let $\calU_1\rightarrow \calU$ be the flat family described in the previous paragraph. Assume that the induced map $\calU\rightarrow \Hilb(\gothA_g/\bA_g)$ for $\calU_1\rightarrow \calU$ is generically finite. 
	 
	 Then the image $\calD_n(\calU^{n+1})$ of the restricted $n$-th Faltings--Zhang morphism   $\calD_n:\calU^{n+1}\rightarrow  \calA^n$ is irreducible, and for $n\gg1$, the image $\calD_n(\calU^{n+1})\subseteq \calA^n$ is non-degenerate.
\end{prop}

Since there are many morphisms involved in the proof below while there is only one  morphism interesting to us between any two varieties, we decide to refer to most morphisms by the arrow.
\begin{proof}
The irreducibility follows from the irreducibility of $\calU^{n+1}$ by the integrality assumption on $S$ and $\calU_{\bar\eta}$, as in the proof of Lemma \ref{lemma_hilb}. 

	 As in Proposition \ref{nondegeneracy}, we aim to show that the moduli map $\calD_n(\calU^{n+1})\rightarrow \gothA^n_g$ is generically finite when $n$ is large.
	 
	  Write $(P,s)$ for a point of $P\in \calU_s(\bar\QQ)$ with $s\in S(\bar\QQ)$. Note that 
	\[
	\calD_n(\calU^{n+1})(\bar\QQ)=\{(s,P_1-P_0,\dots,P_n-P_0):s\in S(\bar\QQ),P_0,\dots,P_n\in \calU_s(\bar\QQ)\}.
	\]
	 Note also that
	\[
 \begin{split}
	(\calU_1)_\calU^n(\bar\QQ)&=\{(s,P_0,P_1-P_0,\dots,P_n-P_0):s\in S(\bar\QQ),P_0,\dots,P_n\in\calU_s(\bar\QQ)\}\\
 &=\{(s,P_0;P_1-P_0,\dots,P_n-P_0):(s,P_0)\in\calU(\bar\QQ),P_1,\dots,P_n\in \calU_s(\bar\QQ)\}.
 \end{split}
	\] There is a modular map $\calD_n(\calU^{n+1})\rightarrow \gothA_g^n$ induced by $\calA^n\rightarrow \gothA_g^n$ and a modular map $(\calU_1)_{\calU}^n\rightarrow\gothA_g^n$ induced by $\calA^n_\calU\rightarrow\gothA_g^n$. They are related by the map $(\calU_1)_\calU^n\rightarrow\calD_n(\calU^{n+1})$ forgetting the $P_0$-component.
	
   Note that $\calU$ is irreducible as a special case by the first paragraph. The integrality of the geometric generic fiber of $\calU_1\rightarrow \calU$ follows from the fact that $\calU_1\cong\calU\times_S\calU$ as $\calU$-schemes, where $\calU\times_S\calU$ is a scheme over $\calU$ via the first projection. If the induced map $\calU\rightarrow \Hilb(\gothA_g/\bA_g)$ for $\calU_1\rightarrow \calU$ is generically finite, then by Corollary \ref{generically_finite_is_ok}, for $n\gg1$ we know the modular map $(\calU_1)_\calU^n\rightarrow \gothA_g^n$ is generically finite. Because $(\calU_1)_\calU^n\rightarrow \gothA_g^n$ factors surjectively through the other map $\calD_n(\calU^{n+1})\rightarrow \gothA_g^n$, we have the generically finiteness of $\calD_n(\calU^{n+1})\rightarrow \gothA_g^n$ by irreducibility. Note that $\gothA_g^n\rightarrow\gothA_{gn}$ is quasi-finite. So the composition map $\iota:\calD_n(\calU^{n+1})\rightarrow\gothA_{gn}$ is generically finite.

	Applying Gao's Betti rank formula \cite[Theorem 10.1(ii)]{Gao_Bettirank} (see also \cite{GaoErratum}) with $t=0$, the generic Betti rank of $\calD_n(\calU^{n+1})$ is at least $2\dim \iota(\calD_n(\calU^{n+1}))$ for $n\geq \dim \calU$. 
	 By the previous paragraph, $\dim\iota(\calD_n(\calU^{n+1}))=\dim \calD_n(\calU^{n+1})$ if $n\gg1$. So for $n\gg1$, the subvariety $\calD_n(\calU^{n+1})\subseteq \calA^n$ is non-degenerate.
\end{proof}

\begin{remark}
	The assumption on the generically finiteness of $\calU\rightarrow \Hilb(\gothA_g/\bA_g)$ is a fairly restrictive condition.  It does not allow the family to have a positive dimensional family of translates of a same subvariety. For example, it excludes the extreme case when $\calU\rightarrow S$ is the family of translates of some fixed variety $X$ in the single abelian variety $A$. On the other hand, if we take $\calU\rightarrow S$ as the universal curve $\gothC_g\rightarrow \bM_g$ where $\bM_g$ is the fine moduli space of curves of genus $g$ with level-$l$-structure,  then since the fibers are essentially distinct, we get the non-degenerate subvariety $\calD_n(\gothC_g^{n+1})$ which is used in \cite{DGH21}.
	\end{remark}

\section{Previous Results}
\subsection{Classical results}\label{subsection_classicalresults}
In this subsection, we recall some classical results developed by R\'emond. 

Let us start by introducing two fundamental constants $c_{\mathrm{NT}}$ and $h_1$. Suppose $A$ is an abelian variety over $\bar\QQ$ with a symmetric very ample line bundle $L$. To use R\'emond's result, we assume moreover that a basis of $\mathrm{H}^0(A,L)$ gives rise to a projectively normal embedding $A\hookrightarrow \PP^n_{\bar\QQ}$ for some $n\in \NN$. This is the case if $L$ is an at least fourth power of a symmetric ample line bundle; see \cite[Theorem 9]{Mumford_quadratic}.

Let $X\subseteq A$ be an integral subvariety. Let $\deg X$ denote the degree of $X$ and let $h(X)$ denote the height of $X$, both considered as a subvariety of $\PP^n_{\bar\QQ}$.

The first constant $c_{\mathrm{NT}}$ is an upper bound for the difference between the N\'eron--Tate height $\hat h:A(\bar\QQ)\rightarrow \RR_{\geq0}$ and the naive logarithmic Weil height $h:\PP^n(\bar\QQ)\rightarrow \RR_{\geq0}$ on $A(\bar\QQ)$. Namely $c_{\mathrm{NT}}$ satisfies $|\hat h(P)-h(P)|\leq c_{\mathrm{NT}}$ for all $P\in A(\bar\QQ)$.
 
The second constant $h_1$ is a measure for the heights of the bihomogeneous polynomials that define the addition and subtraction on $A$. See \cite[\S5]{remond-inegalite} for more details.

Both constants $c_{\mathrm{NT}}$ and $h_1$ are related to $(A,L)$ and the choice of the basis of $H^0(A,L)$ for defining the embedding $A\hookrightarrow \PP^n$. In fact, it is known that  there exists some choice of a basis of $H^0(A,L)$ for the embedding $A\hookrightarrow\PP^n$ and $c'=c'(\dim A,\deg A)$:
\begin{equation}\label{bound on constants}
    c_{\mathrm{NT}},h_1\leq c' \max\{1,\hfal(A)\},
\end{equation}where $\hfal(A)$ denotes the Faltings height of $A$. More concretely, consider the universal family of abelian varieties $\gothA_g\rightarrow\bA_g$ of dimension $g$ with polarization of degree $(l/g!)^2$ and some sufficiently high level structure of level coprime to the degree of the polarization and apply the arguments in \cite[(8.4)(8.7)]{DGH21}. A standard comparison result between the Faltings height and the moduli height gives a bound for any $s\in \bA_g(\bar\QQ)$
 \begin{equation}\label{comparison of heights}
c_0'\max\{1,h(s)\}\leq \max\{1,\hfal(\gothA_{g,s})\}\leq c_0\max\{1,h(s)\}
 \end{equation}where $c_0,c_0'$ are constants only depending on our choice of height on $\bA_g$; see for example \cite[Th\'eor\`eme 1.1]{MB85} together with \cite{Faltings83} to deal with logarithmic singularities.

Write $|P|:=\hat h(P)^{1/2}$ for $P\in A(\bar\QQ)$. Let $C, X\subseteq A$ be integral GeM subvarieties of dimension one and two respectively. We record the following explicit Vojta's and Mumford's inequalities by R\'emond for dimension one and two, which are special cases of \cite[Th\'eor\`eme 1.2]{remond-inegalite} and \cite[Proposition 3.4]{remond-decompte}, respectively. Note that increasing the constants in them will only weaken the results.

	\begin{lemma}[Vojta's inequality {\cite{remond-inegalite}}] \label{Vojta's_inequality}
 \begin{enumerate}
     \item  There exists a constant $c=c(n,\deg C)>1$ such that if $P_1,P_2\in C(\bar\QQ)$ satisfy
		\[
		\begin{split}
			\langle P_1,P_2  \rangle &\geq (1-1/c)|P_1||P_2|\\
			|P_2|&\geq c|P_1|,
		\end{split}
		\]
		then $|P_1|^2\leq c\max\{1,h(C),h_1,c_{NT}\}$.
    \item  There exists a constant $c=c(n,\deg X)>1$ such that if $P_1,P_2,P_3\in X(\bar\QQ)$ satisfy
		\[
		\begin{split}
			\langle P_i,P_{i+1}  \rangle &\geq (1-1/c)|P_i||P_{i+1}|\\
			|P_{i+1}|&\geq c|P_i|
		\end{split}
		\]for $i=1,2$,
		then $|P_1|^2\leq c\max\{1,h(X),h_1,c_{NT}\}$.
 \end{enumerate}
	\end{lemma}
	
	\begin{lemma}[Mumford's inequality {\cite{remond-decompte}}]\label{Mumford's_inequality}
 \begin{enumerate}
     \item Assume that $P_1-P_2\notin \Stab(C)$. There exists a constant $c=c(n,\deg C)>1$ such that if $P_1,P_2$ satisfy
            \[
		\begin{split}
			\langle P_1,P_2  \rangle &\geq (1-1/c)|P_1||P_2|\\
			\big||P_2|-|P_1|\big| &\leq \frac{1}{c}|P_1|,
		\end{split}
		\]
		then $|P_1|^2\leq c\max\{1,h(C),h_1,c_{NT}\}$.
     \item Assume that $(P,P_1,P_2)\in X^3(\bar\QQ)$ is an isolated point in the fiber of the restricted Faltings--Zhang morphism $\calD_2:X^3\rightarrow A^2$. There exists a constant $c=c(n,\deg X)>1$ such that if $P,P_1,P_2$ satisfy
			\[
		\begin{split}
			\langle P,P_{i}  \rangle &\geq (1-1/c)|P||P_i|\\
			\big||P|-|P_i|\big| &\leq \frac{1}{c}|P|
		\end{split}
		\] for $i=1,2$,
		then $|P|^2\leq c\max\{1,h(X),h_1,c_{NT}\}$.
     
 \end{enumerate}
	\end{lemma}

  The following proposition deals with the large points for curves. Note that the first alternative trivially implies the second alternative if $C(\bar\QQ)\cap\Gamma\neq\emptyset$.

\begin{prop}[R\'emond, curve case]\label{large_points_dim_1}
Let $C\subseteq A$ be as in the above. Write $g=\dim A, d=\deg C, l=\deg A$. There exist positive constants $c_3=c_3(g,d,l)$ and $c_4=c_4(g,d,l)$ such that for any finite rank subgroup $\Gamma\subseteq A(\bar\QQ)$ of rank $r$, either
	\[
	\#\left( C(\bar\QQ)\cap \Gamma\right)\leq c_3^{1+r},
	\]or there exists $Q_0\in C(\bar\QQ)\cap\Gamma$ such that 
	\[
	\#\{P\in C(\bar\QQ)\cap\Gamma:\hat h(P-Q_0)> c_4\max \{1,\hfal(A)\}\}\leq c_4^{1+r}.
	\]
\end{prop}
\begin{proof}
Take $c=c(n,d)=c(g,d,l)>0$ from Vojta's inequality \ref{Vojta's_inequality} and Mumford's inequality \ref{Mumford's_inequality} in the curve case. Let $c'=c'(g,l)$ be as in \eqref{bound on constants}. Enlarging $c$ to $c\cdot \max\{1,c'\}$, we can change the conclusions in both inequalities to
\begin{equation}\label{eqn_large}
    |P_1|^2\leq c\max \{1,h(C),\hfal(A)\}.
\end{equation}
By a standard packing argument (see \cite[Corollaire 6.1]{remond-decompte}), we can cover the $r$-dimensional vector space $\Gamma\otimes\RR$ by at most $(1+\sqrt{8c})^r$ cones such that for any $x,y$ in each cone, one has
\[
\langle x,y \rangle\geq (1-1/c)|x||y|.
\]Take one cone and call it $D$. Let us count the number of large points in $C(\bar\QQ)\cap D$ not satisfying \eqref{eqn_large}. Mumford's inequality says that for any $P_1,P_2\in C(\bar\QQ)\cap D$ with $P_1-P_2\notin \Stab (C), |P_1|\leq |P_2|$, we have $|P_2|>(1+1/c)|P_1|$. So the points are discrete and we can find a point $P_1\in C(\bar\QQ)\cap D$ with the smallest height.  Vojta's inequality says that all points in $C(\bar\QQ)\cap D$ satisfy $|P|<c|P_1|$. In particular, the set $C(\bar\QQ)\cap D$ is finite. Order the set by the heights $|P_1|<|P_2|<\dots<|P_N|$. Then clearly $N\leq \log_{1+1/c}(c)$. In particular, we get
\begin{equation}\label{large points with hc}
 \begin{split}
     \#\left\{P\in C(\bar\QQ)\cap\Gamma:\hat h(P)>c\max \{1,h(C),\hfal(A)\}\right\}&\leq \#\Stab(C)\cdot \log_{1+1/c}(c)\cdot (1+\sqrt{8c})^r\\
     &\leq d^2 \log_{1+1/c}(c)\cdot (1+\sqrt{8c})^r.
 \end{split}   
\end{equation}

We need the following lemma of R\'emond, which allows us to eventually remove the dependency on the height of $C$. Note that it is a special case of \cite[Lemme 3.1]{remond-decompte} in dimension one with $n=l/g!-1$.
\begin{lemma}\label{lemma_removing_hx}
	 If $\Xi\subseteq C(\bar\QQ)$ is a finite set of points of cardinality at least $ld^2/g!+1$, then 
	\[
	h(C)\leq d(l/g!+1)^2\cdot \left( \max_{P\in \Xi} \hat h(P) + c_{\mathrm{NT}} +3\log(l/g!)  \right).
	\]
\end{lemma}

For any $Q\in \Gamma$, by \eqref{large points with hc}, there is a smallest constant $\delta_Q$ such that 
\[
\#\{P\in C(\bar\QQ)\cap \Gamma: \hat h(P-Q)>\delta_Q\}\leq \log_{1+1/c}(c)\cdot (1+\sqrt{8c})^{r}.
\]This can be easily seen from the fact that $\hat h(P-Q)$ will be dominated by $\hat h(P)$ if $\hat h(P)$ is large enough. Another way to see it is by noticing the bijection
\[
\{P\in C(\bar\QQ)\cap \Gamma: \hat h(P-Q)>\delta_Q\}\leftrightarrow \{P\in \left(C(\bar\QQ)-Q\right)\cap\Gamma:\hat h(P)>\delta_Q\}.
\] In particular, we see that 
\[
\delta_Q\leq c\max\{1,h(C-Q),\hfal(A)\}.
\]Pick $Q_0$ from the finite set $C(\bar\QQ)\cap\Gamma$ such that $\delta_{Q_0}$ is the smallest among all $\{\delta_Q\}_{Q\in C(\bar\QQ)\cap\Gamma}$. We are going to show that if
$\{P\in C(\bar\QQ)\cap\Gamma:\hat h(P-Q_0)\leq \delta_{Q_0}\}$ has too many points, then Lemma \ref{lemma_removing_hx} will imply that $h(C-Q_0)$ is comparable to $\max\{1,\hfal(A)\}$.

Take $\lambda:=c\cdot d(l/g!+1)^2$. By the packing lemma \cite[Lemme 6.1]{remond-decompte}, we can cover the $r$-dimensional ball centered at $Q_0$ of radius $\sqrt{\delta_{Q_0}}$ using $(1+4\sqrt{2\lambda})^r$ small balls of radius $\sqrt{\delta_{Q_0}/2\lambda}/2$. If there is a small ball $B$ which contains at least $ld^2/g!+1$ points, then by Lemma \ref{lemma_removing_hx}, for any $Q$ in it, we have
\[
\begin{split}
    \delta_Q&\leq c\max\{1,h(C-Q),\hfal(A)\}\leq c\max\{1,\hfal(A)\}+c\cdot h(C-Q)\\&\leq [c+3d(l/g!+1)^2\log(l/g!)]\max\{1,\hfal(A)\}+\lambda \max_{P\in B}\hat h(P-Q)\\
    &\leq [c+3d(l/g!+1)^2\log(l/g!)]\max\{1,\hfal(A)\} +  \frac{1}{2}\delta_{Q_0}.
\end{split}
\]Combining it with $\delta_{Q_0}\leq\delta_Q$, we get
\[
\delta_{Q_0}\leq 2[c+3d(l/g!+1)^2\log(l/g!)]\max\{1,\hfal(A)\}.
\]Then we are done in the second alternative with 
\[
c_4=\max \left\{ 2[c+3d(l/g!+1)^2\log(l/g!)] , \log_{1+1/c}(c), 1+\sqrt{8c} \right\}.
\]
If all the balls contain at most $ld^2/g!$ points, then
\[
\#(C(\bar\QQ)\cap\Gamma)\leq \log_{1+1/c}(c)\cdot (1+\sqrt{8c})^r+ (1+4\sqrt{2\lambda})^r\cdot (ld^2/g!)\leq c_3^{1+r}
\]for some $c_3=c_3(g,d,l)$ large enough.

\end{proof}
	

\subsection{Recent results}
The first theorem we recall is the height inequality for a non-degenerate subvariety, as proved by Dimitrov--Gao--Habegger in \cite[Theorem 1.6 and Appendix B]{DGH21}.
\begin{theorem}[Height inequality {\cite{DGH21}}]\label{height_inequality}
	Let $S$ be an irreducible quasi-projective variety over $\bar\QQ$ and let $\calA\rightarrow S$ be a projective abelian scheme. Fix a relatively ample symmetric line bundle $\calL$ on $\calA\rightarrow S$, which induces fiber-wise N\'eron--Tate heights $\hat h:\calA(\bar\QQ)\rightarrow\RR_{\geq0}$. Fix any height function $h:S(\bar\QQ)\rightarrow \RR$ restricted from a compactification $\bar S$ of $S$ associated to a not necessarily ample line bundle on $\bar S$. Let $\calX$ be an irreducible closed subvariety of $\calA$ that dominates $S$. Suppose $\calX$ is non-degenerate. Then there exist constants $c_1>0,c_2\geq 0$ and a Zariski dense open subset $\calU$ of $\calX$, depending on all the information above, such that
	\[
	\hat h(P)\geq c_1 h(s)-c_2\qquad\text{for all }P\in \calU_s(\bar\QQ).
	\]
\end{theorem}

Note that we allow some extra flexibility of the height function on $S$ in the statement. This easily follows from the original statement since any height function is bounded above by a constant multiple of an ample height function.

The next result is the uniform Bogomolov conjecture for curves. The case of curves embedded in their Jacobians is proved by K\"uhne \cite[Proposition 22]{Kuhne}. To deal with more general curves inside abelian varieties, we prove the following version using the non-degeneracy result in \S\ref{section_nondegeneracy}. A more detailed proof is given in Appendix \ref{Appendix B}.
 
\begin{theorem}[Uniform Bogomolov conjecture {\cite{Kuhne}}]\label{Kuhne_Bogomolov}
	Let $S$ be an irreducible variety over $\bar\QQ$ and let $\calA\rightarrow S$ be a principally polarized abelian scheme with level-$l$-structure. Let $\calC\subseteq\calA$ be an integral closed subvariety whose fibers over the closed points of $S$ are GeM integral curves. Assume moreover that for any $s\in S(\bar\QQ)$, the curve $\calC_s$ generates $\calA_s$.
	Then, there exist constants $c_1,c_2>0$, depending on all the information above, such that for any $s\in S(\bar\QQ)$, we have
	\[
	\#\{x\in\calC_s(\bar\QQ):\hat h(x)\leq c_1\}< c_2,
	\]
	where the fiber-wise N\'eron--Tate heights $\hat h$ are induced from those on the universal family $\gothA_g\rightarrow\bA_g$, as fixed in \S\ref{section_notation}.
\end{theorem}

\begin{proof}
	We explain here the new input to modify the proof in \cite[Proposition 8.3]{Gao_survey}. We may assume that $S$ is reduced, since pulling back the families to $S_{\text{red}}$ does not change the statement.

	Induct on $\dim S$. The case $\dim S=0$ is the classical Bogomolov conjecture.
	
	We may assume $S$ is quasi-projective and smooth and $\calC\rightarrow S$ is flat. So we have an induced morphism $\phi:S\rightarrow\Hilb(\gothA_g/\bA_g)$. We may replace $S$ by its image without loss of generality, since our N\'eron--Tate heights are taken from the universal family. So we assume $\phi$ is generically finite. 
	
	To prove the results on $\calC\rightarrow S$, we remark that the rest of the proof in \cite[Theorem 8.3]{Gao_survey} is valid except that the nondegeneracy of $\calC^n$ is provided by Proposition \ref{nondegeneracy}. Indeed, the geometric generic fiber $\calC_{\bar\eta}$ is integral since the geometrically integral locus is constructible by \cite[Th\'eor\`eme 9.7.7]{EGAIV} and contains all the closed points of $S$. Also, the fact that $\calC_{\bar\eta}$ has a finite stabilizer in $\calA_{\bar\eta}$ and generates $\calA_{\bar\eta}$ follows from a standard spreading-out argument.

\end{proof}

\section{Proof of Theorem {\ref{curve}}}

\begin{proof}[Proof of Theorem {\ref{curve}}]
Induct on $\dim S$. The case $\dim S=0$ is a classical result of R\'emond \cite[Th\'eor\`eme 1.2]{remond-decompte}. Consider $\dim S>0$.

We arrange the proof in a few steps. Remark that although in the theorem we allow the fibers $\calC_s$ to be not of pure dimension, the argument below first reduces to the case where fibers are integral and then ignores the case where fibers are points, since they are trivial.

\boxed{\text{Step 0}} We make several reductions in this step. 

Without loss of generality, assume $\calC$ and $S$ are integral. Note first the following simple fact by using the induction hypothesis and Chevalley's theorem on constructibility of the image:

\emph{If $S'\rightarrow S$ is a dominant map, then the result on $\calC_{S'}\subseteq \calA_{S'}$} implies the result on $\calC\subseteq \calA$.

Immediate consequences of the fact are that we may assume $S$ is \emph{quasi-projective} or even affine, and assume that the map $\calC\rightarrow S$ is \emph{flat}.

We will first add polarization and level structure to the abelian scheme. After possibly replacing $S$ by a finite cover, let $(A',L')$ be a principally polarized abelian variety isogenous to the generic abelian variety $\calA_\eta$. By spreading out $(A',L')$ and the morphism, there is a principally polarized abelian scheme $(\calA',\calL')$ over a dense open subset $S'$ of $S$ with an $S'$-isogeny $\calA'\rightarrow \calA\times_S S'$. Let $\calC'$ be the pullback of $\calC$ to $\calA'$. Then the result on $\calC'\subseteq\calA'$ would imply the result on $\calC\subseteq\calA$ over $S'$. In particular, we may assume \emph{$\calA\rightarrow S$ is a principally polarized abelian scheme}.

For the level structure, note that there is a quasi-finite dominant map $S'\rightarrow S$ such that the $l$-torsion points of the generic abelian variety $\calA_\eta$ are all defined over $k(S')$ and the torsion points spread out to distinct torsion sections over $S'$. Then we may impose the level-$l$-structure on $\calA\times_S S'$. By the fact, we thus assume \emph{$\calA\rightarrow S$ is a principally polarized abelian scheme with level-$l$-structure}.

 The generic fiber $\calC_\eta$ is irreducible, but may not be geometrically irreducible. There is a finite extension of the function field $k(S)$, which gives rise to a quasi-finite dominant map $T\rightarrow S$ for some integral variety $T$, such that the generic fiber of the pullback $\calC_T\rightarrow T$ decomposes into geometrically irreducible components. By the fact, we may work on $\calC_T\rightarrow T$ instead. Then we can work on each irreducible component of $\calC_T$. Let $\calC_1\subseteq\calC_T$ be an irreducible component. It is dominant over $T$. Then the generic fiber of $\calC_1$ is an irreducible component of the generic fiber of $\calC_T$. So the generic fiber of $\calC_1\rightarrow T$ is geometrically irreducible. By \cite[\href{https://stacks.math.columbia.edu/tag/0559}{Lemma 0559}]{stacks-project}, there exists a dense open subset $T'$ of $T$ such that $\calC_1\times_T T'\rightarrow T'$ has geometrically irreducible fibers. Together with induction hypothesis, we reduce to study such a family as $\calC_1\times_T T'\rightarrow T'$ and may assume \emph{every fiber is geometrically irreducible}. 
 
 Remark that the quasi-projectivity of $S$, integrality of $\calC,S$ and flatness can always be re-obtained by further trimming the family (without affecting the geometrically irreducibility of fibers). In fact, in all the remaining modifications, individual fibers will remain untouched up to translation.

 The generic fiber $\calC_\eta$ is clearly reduced (and hence geometrically reduced since we are in characteristic zero), since we assumed $\calC$ is reduced and $\calC_\eta$ is a localization. Then \cite[\href{https://stacks.math.columbia.edu/tag/0578}{Lemma 0578}]{stacks-project} tells us that over a dense open $V$ of $S$, all the fibers are reduced. So we may assume \emph{every fiber is geometrically reduced}.

We can assume \emph{the curves pass through $0$}. Indeed, by possibly taking a finite cover, which is justified by the fact above, we can assume there is a point $\Spec k(S)\rightarrow \calC_\eta$ on the generic fiber. By spreading out the morphism, there is a dense open subset $S'$ of $S$ such that $S'\rightarrow \calC':=\calC\times_S S'$ is a section. By the fact, it suffices to work on $\calC'\rightarrow S'$. Translating $\calC'$ by the inverse of the section, we get $\calC''\rightarrow S'$. The result on $\calC''\rightarrow S'$ implies the result on $\calC'\rightarrow S'$, since $\calC'_s(\bar\QQ)\cap\Gamma\hookrightarrow (\calC'_s(\bar\QQ)-P)\cap \langle\Gamma, P\rangle$, where $P\in \calC_s'(\bar\QQ)$, $\Gamma\subseteq \calA_s(\bar\QQ)$ and $\langle\Gamma, P\rangle$ is the group generated by $\Gamma$ and $P$, which has rank at most $1+\rank\Gamma$.

We may assume \emph{the curves generate the ambient abelian varieties}. Indeed, if the geometric generic curve $\calC_{\bar\eta}$ generates $\calA_{\bar\eta}$, then the addition map $+:(\calC-\calC)^g\rightarrow \calA$ is surjective over the geometric generic point of $S$. By properness, $+$ is surjective. Namely, the curves over $s\in S$ generate the ambient abelian varieties. Otherwise, by our assumption that the fibers of $\calC\rightarrow S$ are geometrically irreducible, the fibers of $(\calC-\calC)^g\rightarrow S$ are also geometrically irreducible. Then the image of the addition $+:(\calC-\calC)^g\rightarrow \calA$ is a closed subscheme $\calB$ of $\calA$, which is closed under group operations. So the geometric generic fiber of $\calB\rightarrow S$ is a proper connected algebraic subgroup of $\calA_{\bar\eta}$, hence an abelian subvariety. By spreading out again and discarding a lower dimensional closed subset of $S$, we may assume $\calB\rightarrow S$ is an abelian subscheme. Assume we have the result for $\calC\subseteq \calB\rightarrow S$, namely there is $c>0$ such that for any $s\in S(\bar\QQ)$ and any subgroup $\Gamma$ of $\calB_s(\bar\QQ)$ of finite rank $r$, we have $\#(\calC_s(\bar\QQ)\cap \Gamma)\leq c^{1+r}$. Then the result for $\calC\subseteq\calA\rightarrow S$ will be an immediate corollary because for any finite rank subgroup $\Gamma\subseteq\calA_s(\bar\QQ)$, the intersection $\Gamma\cap\calB_s(\bar\QQ)$ is a subgroup of possibly lower rank, which means the same $c>0$ would work for $\calC\subseteq \calA$.

Overall, we make the following assumptions:
\begin{enumerate}[topsep=0pt]
	\item $S$ is an integral quasi-projective variety over $\bar\QQ$;
	\item $\calA\rightarrow S$ is a principally polarized abelian scheme with level-$l$-structure;
	\item $\calC\rightarrow S$ is a flat family of GeM integral curves passing through $0$ in $\calA\rightarrow S$, such that every curve $\calC_s$ generates $\calA_s$.
\end{enumerate}
In particular, there is a natural modular map $\iota_S$ from $S$ to the moduli space $\bA_g$ of principally polarized abelian varieties with level-$l$-structure by (2), and a natural modular map $\psi$ from $S$ to the Hilbert scheme $\Hilb(\gothA_g/\bA_g)$ by (3). The map $\iota_S$ induces fiber-wise N\'eron--Tate heights on $\calA\rightarrow S$ by pulling back those on the universal family. Recall that we defined the height function $h:\bA_g(\bar\QQ)\rightarrow \RR$ in \eqref{section_notation_abelianschemes}.  By abuse of notation, we still use $h$ to denote the height
function on $S(\bar\QQ)$ given by $h(s):=h(\iota_S(s))$.

\boxed{\text{Step 1}}  In this step, we will demonstrate how to apply the height inequality \ref{height_inequality} to the non-degeneracy result from \S\ref{section_nondegeneracy} to study the distance of algebraic points on the curves. 

If $\psi: S\rightarrow \Hilb(\gothA_g/\bA_g)$ is not generically finite, then the image has lower dimension. In this case we can simply use induction on the image family and pull it back. So we can and do assume $\psi: S\rightarrow \Hilb(\gothA_g/\bA_g)$ is generically finite.

Consider the following embedding of $\calC\times_S\calC$:
	\[
	\begin{split}
		\calC\times_S\calC&\longrightarrow \calA_\calC=\calC\times_S\calA\\
		(P,Q)&\longmapsto (P,Q-P).
	\end{split}
	\]
	Denote its image by $\calC_1$. Then the projection onto the first factor is a flat family of curves $\calC_1\rightarrow \calC$ inside $\calA_{\calC}$. Note that each fiber of $\calC\rightarrow S$ is a fiber of $\calC_1\rightarrow \calC$, due to our assumption that $0\in \calC_s(\bar\QQ)$. By the universal property of Hilbert schemes, there is a natural modular map $\phi:\calC\rightarrow\Hilb(\gothA_g/\bA_g)$. Then $\psi(S)\subseteq\phi(\calC)$. Depending on whether $\phi$ is generically finite or not, we apply Theorem \ref{height_inequality} in two ways. 
	
	If $\phi$ is generically finite, we use the Faltings--Zhang morphism to study the distance. By Proposition \ref{nondegeneracy_for_FZ}, for $n\gg1$, the image of the Faltings--Zhang morphism $\calD_n(\calC^{n+1})\subseteq\calA^n$ is non-degenerate. Fix such an $n$. By Theorem \ref{height_inequality}, there exist $c_1,c_2>0$ (with an extra dependence on $n$) and a dense open subset $V$ of $\calD_n(\calC^{n+1})\subseteq\calA^n$ such that
	\begin{equation}\label{height:00}
		\hat h(P_1-P_0)+\dots+\hat h(P_n-P_0)\geq n\cdot(c_1 h(s)-c_2)
	\end{equation}
	for any $s\in S(\bar\QQ)$ and $(P_1-P_0,\dots,P_n-P_0)\in V_s(\bar\QQ)$.
		
	If $\phi$ is not generically finite, then the closure of the scheme-theoretic images of $\psi:S\rightarrow \Hilb(\gothA_g/\bA_g)$ and $\phi:\calC\rightarrow \Hilb(\gothA_g/\bA_g)$ must coincide, since $\psi(S)\subseteq\phi(\calC)$ and they are both integral of the same dimension. By Proposition \ref{nondegeneracy}, for $n\gg1$, the fibered power $\calC^n\subseteq \calA^n$ is non-degenerate. Fix such an $n$. By Theorem \ref{height_inequality}, there exist $c_1,c_2>0$ (with an extra dependence on $n$) and a dense open subset $V'\subseteq\calC^n$ such that
	\begin{equation}\label{height:0}
		\hat h(P_1)+\dots+\hat h(P_n)\geq n\cdot (c_1 h(s)-c_2)
	\end{equation}
	for any $s\in S(\bar\QQ)$ and $(P_1,\dots,P_n)\in V_s'(\bar\QQ)$.

	\boxed{\text{Step 2}} For the rest, let us say two points $P_1,P_2\in \calA_{s}(\bar\QQ)$ with $s\in S(\bar\QQ)$ are \emph{close} if $\hat h(P_1-P_2)<c_1h(s)-c_2$. 
	
	Define $E_n:=\{(P,Q_1,\dots,Q_n)\in \calC^{n+1}(\bar\QQ): Q_i\text{ is close to }P\text{ for any }i\}$ and let $F_n$ be the Zariski closure of $E_n$ in $\calC^{n+1}$. Claim: $F_n$ is a proper subset of $\calC^{n+1}$. 
	
	Indeed, for the first case where $\phi$ is generically finite, simply notice that the image of a tuple $(P,Q_1,\dots,Q_n)\in E_n$ under $\calD_n$ is not in $V$ by \eqref{height:00}. 
	
	 For the other case, since up to a finite base change and away from a proper Zariski closed subset, $\calC_1\rightarrow \calC$ is a pullback of $\calC\rightarrow S$ (the morphism $\calC\rightarrow S$ inducing the pullback does not have to be the same as the structural morphism $\calC \rightarrow S$), the open subset $V'\subseteq\calC^n$ pulls back to an open dense subset $V''$ of $(\calC_1)_\calC^n$ with the same height inequality as in (\ref{height:0}) because of our choice of the height function on $S$ at the end of Step $0$.  Namely, we have
	\begin{equation}
		\hat h(P_1)+\dots+\hat h(P_n)\geq n\cdot(c_1 h(s)-c_2)
	\end{equation} for any $s\in S(\bar\QQ)$, $P\in \calC_s(\bar\QQ)$ and $(P_1,\dots,P_n)\in V_{s,P}''(\bar\QQ)$. Notice that a point $(P_1,\dots,P_n)\in (\calC_1)_{\calC}^n(\bar\QQ)$ over $P\in \calC_s(\bar\QQ)$ is of the form $(Q_1-P,\dots,Q_n-P)$ for some $(Q_1,\dots,Q_n)\in \calC^n_s(\bar\QQ)$. Then it is clear that $E_n$ is not Zariski dense in $\calC^{n+1}$ through the following identification
	\[
	\begin{split}
		\calC^{n+1}&\longrightarrow (\calC_1)_\calC^n\\
		(s;P,Q_1,\dots,Q_n)&\longmapsto (s,P;Q_1-P,\dots,Q_n-P),
	\end{split}
	\]since a point of $E_n$ corresponds to a point outside $V''$. 
	
	 Let $V_n:=\calC^{n+1}\setminus F_n$. The image of $V_n$ in $S$, say $S_0$, is Zariski open by flatness. We can use induction hypothesis on $S\setminus S_0$. So without loss of generality, assume the projection $V_n\rightarrow S$ is surjective. Then we have the following lemma:

	\begin{lemma}\label{Lemma_curve}
		There exist $N_1,N_2>0$ such that for any $s\in S(\bar\QQ)$, we have
		\[
		\#\{P\in\calC_s(\bar\QQ):\text{there are }>N_1\text{ points close to }P\}\leq N_2
		\]
	\end{lemma}
	\begin{proof}
	This is essentially the same as \cite[Proposition 7.1]{DGH21}. Here we give an alternative proof using an idea from the proof of \cite[Lemma 1.1]{CHM}.
	
	Define $F_i\subseteq \calC^{i+1}$ inductively in a decreasing order for $0\leq i\leq n$ as follows. The first case $F_n\subseteq\calC^{n+1}$ is defined above the lemma. Assume now that $F_{i+1}\subseteq \calC^{i+2}$ is defined, for $0\leq i \leq n-1$. Let $F_i$ be the largest closed subset of $\calC^{i+1}$ such that $
	\pi_{i+1}^{-1}(F_i)\subseteq F_{i+1}$ where $\pi_{i+1}:\calC^{i+2}\rightarrow \calC^{i+1}$ is the projection leaving out the last component. In other words, $F_i$ is the set of points over which fibers of $\pi_{i+1}|_{F_{i+1}}$ are positive dimensional. It is closed due to upper semi-continuity of fiber dimension  \cite[Corollaire 13.1.5]{EGAIV} and properness of the map $F_{i+1}\rightarrow \calC^{i+1}$. 
	
	For $0\leq i\leq n-1$, let $V_i:=\calC^{i+1}\setminus F_i$. Then $\pi_{i+1}|_{F_{i+1}}$ is finite over $V_i$ and let $d_i$ be an upper bound on the cardinality of the finite fibers. Let $N_1$ be the maximum of all $d_i$'s.
	
	Since $V_n$ surjects to $S$ by assumption above this lemma, $F_i$ is proper in $\calC^{i+1}$, even fiber-wise over any $s\in S$, for any $0\leq i\leq n-1$. In particular, $F_0\subseteq \calC$ maps finitely to $S$. Let $N_2$ be the largest number of points in any fiber of $F_0\rightarrow S$.
	
	Then $N_1,N_2$ satisfy the requirement. Indeed, we will show that if $P\in (V_0)_s(\bar\QQ)$, then there are at most $N_1$ points close to $P$. Let $j(P)\geq 1$ be the smallest integer $j$ such that all points $(P,Q_1,\dots,Q_j)\in\calC^{j+1}_s(\bar\QQ)$ with all $Q_i$'s close to $P$ are contained in $F_j$. Note that $j(P)\leq n$. By definition of $j(P)$, and since $P\in V_0(\bar\QQ)$, there is a point $(P,Q_1,\dots,Q_{j(P)-1})$ with all $Q_i$'s close to $P$ in $V_{j(P)-1}$. But any prolongation $(P,Q_1,\dots,Q_{j(P)-1},Q)$ with $Q$ close to $P$ is in $F_{j(P)}$. Since the fiber of $F_{j(P)}$ over a point of $V_{j(P)-1}$ is finite of cardinality bounded by $N_1$, the number of such $Q$'s that are close to $P$ is bounded by $N_1$. 
	\end{proof}
	
	\boxed{\text{Step 3}}
	Let $\Gamma\subseteq\calA_s(\bar\QQ)$ be a subgroup of finite rank $r$ with $s\in S(\bar\QQ)$. We divide the final step into two cases, depending on whether $h(s)$ is large or small. 
 
 We assume that the line bundle $\calL_0$ is relatively \emph{very ample} and defines a fiber-wise projectively normal embedding $\calA\hookrightarrow\PP^n_S$ for some $n$. The general case with $\calL$ merely relatively ample (and symmetric) follows easily by comparing the heights, since $\calL^{\otimes 4}$ satisfies the required condition by \cite[Theorem 9]{Mumford_quadratic} and notice that $H^0(\calA,\calL)\otimes_{\calO_S(S)} k(s)\cong H^0(\calA_s,\calL_s)$ (see for example \cite[\S5, Corollary 2]{Mumford_AV} and \cite[Proposition 6.13]{GIT}), together with the criterion of projective normality in \cite[Exercise II.5.14(d)]{Hartshorne}. By Proposition \ref{large_points_dim_1} (noticing that $g,d,l$ are bounded in the family), there exist $c_3,c_4>0$, such that either already we have $\#(\calC_s(\bar\QQ)\cap\Gamma)\leq c_3^{1+r}$ and we are done; or there is a $Q_0\in \calC_s(\bar\QQ)$ such that the number of points in $\calC_s(\bar\QQ)\cap\Gamma$ whose distances to $Q_0$ are greater than $\sqrt{c_4\max\{1,\hfal(\calA_s)\}}$ is at most $c_4^{1+r}$.  By the comparison \eqref{comparison of heights}, it suffices to bound the cardinality of
 \[
\left\{P\in \calC_s(\bar\QQ)\cap \Gamma: \hat h(P-Q_0)\leq c_0c_4\max\{1,h(s)\}\right\}.
 \]
	
	\boxed{\text{case a}}	For $s\in S(\bar\QQ)$ such that $h(s)> \max\{1,\frac{2c_2}{c_1}\}$, we use the above to get the desired bound on $\#(\calC_s(\bar\QQ)\cap\Gamma)$. Note first that if two points $P,Q$ in $\calC_s$ satisfies $\hat h(P-Q)\leq\frac{c_1}{2}h(s)$, then 
	\[
	\hat h(P-Q)\leq c_1 h(s)-\frac{c_1}{2}h(s)< c_1 h(s)-c_2.
	\]In other words, $P,Q$ are close.    Then it suffices to bound the number of points $P\in(\calC_s(\bar\QQ)-Q_0)\cap\Gamma$ with
 $\hat h(P)\leq c_0\cdot c_4\max \{1,h(s)\}=c_0 c_4 h(s)$.
 
 Regard 	$\Gamma\otimes \RR$ as a normed vector space with norm given by $|\cdot|:=\hat h^{1/2}$. Let $R_1:=\sqrt{c_0c_4h(s)}$ and $R_2:=\sqrt{\frac{c_1}{2}h(s)}/2$. Assume without loss of generality that $R_1\geq R_2$. The closed ball of radius $R_1$ centered at $Q_0$ is covered by at most $(1+2R_1/R_2)^r=(1+4\sqrt{2c_0c_4/c_1})^r$ closed balls of radius $R_2$ by \cite[Lemme 6.1]{remond-decompte}. In any closed ball of radius $R_2$, any two points are close. So by Lemma \ref{Lemma_curve}, there are at most $N_1$ points in one closed ball of radius $R_2$, with at most $N_2$ exceptional points in total which we may exclude first. Thus overall we have
			\[
			\#(\calC_s(\bar\QQ)\cap\Gamma)\leq \max \left\{c_3^{1+r}, \, c_4^{1+r}+N_1\cdot (1+4\sqrt{2c_0c_4/c_1})^r+N_2 \right\}\leq c^{1+r}
			\]for some large $c>0$.

	\boxed{\text{case b}}
		For $s\in S(\bar\QQ)$ such that $h(s)\leq \max\{1,\frac{2c_2}{c_1}\}$, we use K\"uhne's result to get the desired bound on $\#(\calC_s(\bar\QQ)\cap\Gamma)$. Write 
  \[
  c_5:=c_0c_4\max\left\{1,\frac{2c_2}{c_1}\right\}\geq c_0c_4\max\{1,h(s)\}.
  \] It suffices to bound
   \[
\left\{P\in \calC_s(\bar\QQ)\cap \Gamma: \hat h(P-Q_0)\leq c_5\right\}.
 \]We use Theorem \ref{Kuhne_Bogomolov} for the family $\calC_1\rightarrow\calC$. Note that $\calC_1$ is reduced since $\calC$ is reduced and all fibers are reduced by our reduction step; see \cite[\href{https://stacks.math.columbia.edu/tag/0C21}{Lemma 0C21}]{stacks-project}. Then there exist $c_6,c_7>0$ such that for any $s\in S(\bar\QQ)$ and $P\in \calC_s(\bar\QQ)$,
		\[
		\#\{Q\in\calC_s(\bar\QQ):\hat h(Q-P)\leq c_6\}<c_7.
		\]
		This immediately tells us that a closed ball of radius $\sqrt{c_6}/2$ contains less than $c_7$ points. Now we cover the large ball of radius $\sqrt{c_5}$ centered at $Q_0$ in $\Gamma\otimes \RR$ by at most $(1+4\sqrt{c_5/c_6})^r$ small balls of radius $\sqrt{c_6}/2$ (assuming without loss of generality $2\sqrt{c_5/c_6}\geq 1$). A similar argument as above then yields
		\[
		\#(\calC_s(\bar\QQ)\cap\Gamma)\leq \max\left\{c_3^{1+r}, c_4^{1+r}+c_7\cdot (1+4\sqrt{c_5/c_6})^r\right\}\leq c^{1+r}
		\]for some large $c>0$.

\end{proof}

\section{Proof of Theorem {\ref{mainthm}}}
We will first prove the following family version of R\'emond's result on large points in dimension $2$, for which we need to invoke Theorem \ref{curve}. This is expected to some extent, as R\'emond in his original proof of \cite[Proposition 3.3]{remond-decompte}, uses the most general results for lower dimensions too.

\begin{prop}[R\'emond, surface case]\label{large_points_dim_2}
		Let $S$ be a quasi-projective variety over $\bar\QQ$. Let $\calA\rightarrow S$ be a projective abelian scheme and let $\calX\subseteq\calA$ be an irreducible closed subvariety whose fibers over the closed points of $S$ are GeM integral surfaces. Fix a relatively ample symmetric line bundle $\calL$ on $\calA\rightarrow S$, which induces fiber-wise N\'eron--Tate heights $\hat h:\calA(\bar\QQ)\rightarrow\RR_{\geq0}$. Fix a height function $h:S(\bar\QQ)\rightarrow \RR$ on $S$ corresponding to an immersion of $S$ into a projective space. Then there exists a constant $c>0$ such that for any $s\in S(\bar\QQ)$ and any subgroup $\Gamma\subseteq \calA_s(\bar\QQ)$ of finite rank $r$, we have
	\[
	\#\{P\in \calX_s(\bar \QQ)\cap\Gamma: \hat h(P)>c\max\{1,h(s)\}\}\leq c^{1+r}.
	\]
\end{prop}
\begin{proof}
Induct on $\dim S$. The case $\dim S=0$ follows from classical Vojta's and Mumford's inequalities. Consider $\dim S>0$. Assume $S$ is integral, smooth and affine by induction. Assume $\calX\rightarrow S$ is flat. Assume that the line bundle $\calL$ is relatively \emph{very ample} and  defines a fiber-wise projectively normal embedding $\calA\hookrightarrow\PP^n_S$ for some $n$. The general case follows easily by comparing the heights, since $\calL^{\otimes 4}$ satisfies the condition. See also the second paragraph of Step $3$ of the proof of Theorem \ref{curve}.

Let $\calY\subseteq \calA\times_S\calA$ be the intersection of the subvariety $\calX\times_S\calA$ with the image of the following embedding 
\[
\begin{split}
	\calX\times_S\calA&\hookrightarrow\calA\times_S\calA\\
	(x,a)&\mapsto (x-a,a).
\end{split}
\]Let $p_2:\calY\rightarrow \calA$ be the second projection. Then the fiber of $p_2$ over $a\in \calA_s(\bar\QQ)$ for $s\in S(\bar\QQ)$ is $\calX_s\cap(\calX_s-a)$. Since $\calX_s$ is GeM, the stabilizer of $\calX_s$ is finite. Note that $\calX_s\cap(\calX_s-a)$ is a proper subset of $\calX_s$ if $a$ is not in the stabilizer. Define $\Stab(\calX)$ as the reduced closed subvariety of $\calA$, over which the fibers of $p_2$ are of dimension two. Let $N$ be the largest number of points in any fiber of $\Stab (\calX)\rightarrow S$. For any $s\in S(\bar\QQ)$, for all but at most $N$ points  $a\in \calA_s(\bar\QQ)$, the fiber of $p_2:\calY\rightarrow \calA$ over $a$ is of dimension at most $1$. By the upper-semicontinuity of fiber dimension \cite[Corollaire 13.1.5]{EGAIV}, there exists a dense open subset $V\subseteq \calA$ whose projection to $S$ is surjective, such that the restricted family $\calY':=\calY\times_{\calA}V\rightarrow V$ is of relative dimension at most $1$ and all fibers are GeM. 

If the relative dimension is $0$, we trivially get the same result as in Theorem \ref{curve}. Otherwise, by applying Theorem \ref{curve} to the family of curves $\calY'\rightarrow V$, there exists an integer $c_1>1$ such that for any $s\in S(\bar\QQ)$, and any subgroup $\Gamma\subseteq\calA_s(\bar\QQ)$ of finite rank $r$ and any $a\in V_s(\bar\QQ)$, we have $\#(\calY'_a(\bar\QQ)\cap\Gamma)\leq c_1^{1+r}$. Let $D$ be a uniform upper bound for $\deg(\calX_s)$, with respect to $\calL_s$, for any $s\in S(\bar\QQ)$. We need the following lemma for Mumford's inequality:
	
	\begin{lemma}\label{lemma_isolatedinfiber}
		For any $s\in S(\bar\QQ)$, any $P\in\calX_s(\bar\QQ)$, any $\Gamma\subseteq\calA_s(\bar\QQ)$ of rank $r$ and any subset $\Sigma \subseteq\calX_s(\bar\QQ)\cap \Gamma$ of cardinality  $>N+D^2\cdot c_1^{1+r}$, there exist $P_1,P_2\in\Sigma $ such that $(P,P_1,P_2)$ is an isolated point in the fiber of the restricted Faltings--Zhang morphism $\calD_2:\calX_s^3\rightarrow\calA_s^2$.
	\end{lemma}
	 
	\begin{proof}
	Note that $(P,P_1,P_2)$ is isolated in the fiber over $(P_1-P,P_2-P)$ if and only if the local dimension of $(\calX_s-P)\cap (\calX_s-P_1)\cap (\calX_s-P_2)$ at the identity $0$ is zero. Since $\#\Stab(\calX_s-P)\leq N$, there are at most $N$ points $P_1\in\calA_s(\bar\QQ)$ such that $(\calX_s-P)\cap (\calX_s-P_1)=\calX_s-P$. Pick $P_1\in \Sigma$ such that $\dim ((\calX_s-P)\cap (\calX_s-P_1))\leq 1$. Assume $P_2$ is such that the local dimension of $(\calX_s-P)\cap (\calX_s-P_1)\cap (\calX_s-P_2)$ at the identity $0$ is not zero. Then $\calX_s-P_2$ must contain some irreducible component $0\in C$ of dimension one of $(\calX_s-P)\cap (\calX_s-P_1)$. This implies that 
	\[
	P_2\in \bigcap_{a\in C} (\calX_s-a)=\calX_s\cap \bigcap_{a\in C\setminus \{0\}}(\calX_s-a). 
	\]
	Note that the right hand side is contained in some $\calX_s\cap (\calX_s-a)$  of dimension $1$ for some $a\in (C\cap V_s)(\bar\QQ)$. So there are at most $c_1^{1+r}$ such $P_2\in \Gamma$ by the paragraph above the lemma. On the other hand, the number of such irreducible components $C$ is at most
	\[
	\sum_{C} \deg C\leq \deg(\calX_s)^2\leq D^2.
	\]
	So for $(P,P_1)$, there are at most $D^2\cdot c_1^{1+r}$ choices of $P_2$ such that $(P,P_1,P_2)$ is not isolated in the fiber of the Faltings--Zhang morphism. Hence the claim.
	\end{proof}
	
	Now we are ready to use Vojta's and Mumford's inequalities. Since $\deg\calX_s$ and $n$ are bounded for the family $\calX\rightarrow S$, we can take a uniform constant $c_2>1$ for both the constants obtained in Lemma \ref{Vojta's_inequality} and Lemma \ref{Mumford's_inequality}.
	
	Regard $\Gamma\otimes\RR$ as a normed vector space of dimension $r$. By \cite[Corollaire 6.1]{remond-decompte}, we can cover the vector space by at most $(1+\sqrt{8c_2})^r$ closed cones, such that for any $x,y$ in a same cone, we have $\langle x,y \rangle\geq (1-1/c_2)|x||y|.$
	
	Inside any one of the cones, consider the set of points in $\calX_s(\bar\QQ)\cap\Gamma$ whose N\'eron--Tate heights are greater than 
	\[
	c_2\max\{1,h(\calX_s),h_{1,s},c_{\text{NT},s}\}.
	\]
	Since the set is finite, we may arrange them so that $|P_0|\leq |P_1|\leq \dots$. Then Lemma \ref{Mumford's_inequality} and Lemma \ref{lemma_isolatedinfiber} tells us that for any number $B>c_2\max\{1,h(\calX_s),h_{1,s},c_{\text{NT},s}\}$, there are at most $N+D^2\cdot c_1^{1+r}$ points $P_j$ with 
	$B\leq|P_j|\leq (1+1/c_2)B.$ In particular, we have
	\[
	\left|P_{i+(N+D^2\cdot c_1^{1+r})}\right|>(1+1/c_2)|P_i|, \text{ for any }i.
	\]
	Let $N_1:=\lceil\log_{1+1/c_2}(c_2)\rceil$ so that $(1+1/c_2)^{N_1}\geq c_2$. Write
	\[
	N_2:=(N+D^2\cdot c_1^{1+r})\cdot N_1,
	\]which depends on $r$ as well. Then $|P_{i+N_2}|> (1+1/c_2)^{N_1}|P_i|\geq c_2|P_i|$ for any $i$. Then we cannot have more than $2N_2$ points in the cone, since otherwise the triple $(P_0,P_{N_2},P_{2N_2})$ contradicts Lemma \ref{Vojta's_inequality}.
	
	Overall, we have at most $(1+\sqrt{8c_2})^r$ cones and at most $2N_2$ large points in each cone, so
	\[
	\begin{split}
			\#\{P\in\calX_s(\bar\QQ)\cap \Gamma:\hat h(P)>c_2\max\{1,h(\calX_s),h_{1,s},c_{\text{NT},s}\}\}&\leq (1+\sqrt{8c_2})^r\cdot 2N_2\\
			=(1+\sqrt{8c_2})^r\cdot 2(N+D^2\cdot c_1^{1+r})\cdot N_1 \leq c_3^{1+r}
	\end{split}
	\]for some $c_3>0$. As in \cite[(8.4)(8.7)(8.8)]{DGH21}, we have 
	$h(\calX_s),h_{1,s},c_{\text{NT},s}\leq c_4\max\{1,h(s)\}$ for some $c_4>1$ which does not depend on $s$ and $\Gamma$. So take $c=\max\{c_2c_4,c_3\}$ and we are done.
\end{proof}

Finally we can give the proof of the main theorem of this paper.
\begin{proof}[Proof of Theorem {\ref{mainthm}}]
Induct on $\dim S$. The case $\dim S=0$ holds trivially since we can make $c$ large enough to exclude all the cases. Assume $\dim S>0$. Assume $S$ is integral, smooth and quasi-projective and assume $\calX\rightarrow S$ is flat. Then $\calX^n$ is irreducible since all closed fibers are (geometrically) irreducible and $S$ is irreducible.

By \cite[Theorem 1.3]{Gao_Bettirank} (see \cite{GaoErratum} for the correction), for $m$ large enough, $\calD_m(\calX^{m+1})\subseteq \calA^m$ is non-degenerate. By Theorem \ref{height_inequality}, there is a dense Zariski open subset $U$ of $\calD_m(\calX^{m+1})$ with constants $c_1,c_2>0$, such that for any $(Q_1-P,\dots,Q_m-P)\in U(\bar\QQ)$, we have
	\begin{equation}\label{height:1}
	\sum_i\hat h(Q_i-P)\geq m\cdot(c_1h(\iota_S(s))-c_2),
	\end{equation}where we use the pullback height on $S$ from $\bA_g$ and the N\'eron--Tate heights from the universal abelian scheme.
	We say two points $P,Q\in \calX_s(\bar\QQ)$ for $s\in S(\bar\QQ)$  are close if $\hat h(P-Q)<c_1h(\iota_S(s))-c_2$. By the induction hypothesis, we may moreover assume $U$ surjects onto $S$. We have the following lemma which is similar to Lemma \ref{Lemma_curve}. Remark that an at most curve is a scheme of (relative) dimension at most $1$.

	\begin{lemma}\label{lemma_surface}
		There exist $S$-varieties $S_0=S,S_1,\dots,S_m$ and algebraic families of (at most) curves $\calC_i\subseteq \calX_{S_i}\rightarrow S_i$ with the following property. For any $P\in \calX_s(\bar\QQ)$, one of the following holds
		\begin{enumerate}
			\item either $P\in(\calC_0)_s(\bar\QQ)$, or
			\item the points of $\calX_s(\bar\QQ)$ that are close to $P$ are on an at most curve $C=(\calC_i)_{s'}$ for some $1\leq i\leq m$ and $s'\in S_i(\bar\QQ)$ over $s\in S(\bar\QQ)$.
		\end{enumerate}
	 	\end{lemma}
	Remark that the varieties above can be empty.
	\begin{proof}
	By the inequality (\ref{height:1}), there is a proper Zariski closed subset $F_m$ of $\calX^{m+1}$ such that if $(P,Q_1,\dots,Q_m)\in\calX^{m+1}(\bar\QQ)$ is such that all $Q_i$'s are close to $P$, then $(P,Q_1,\dots,Q_m)\in F_m$. 
	
	Define $F_i$ decreasing-inductively for $0\leq i\leq m-1$ as follows. Assume $F_{i+1}$ is defined. Let $F_i$ be the largest closed subset of $\calX^{i+1}$ such that $
	\pi_{i+1}^{-1}(F_i)\subseteq F_{i+1}$ where $\pi_{i+1}:\calX^{i+2}\rightarrow \calX^{i+1}$ is the projection leaving out the last component. In other words, $F_i$ is the set of points of $\calX^{i+1}$ over which fibers of $\pi_{i+1}|_{F_{i+1}}$ are $2$ dimensional, which is closed due to upper semi-continuity of fiber dimension \cite[Corollaire 13.1.5]{EGAIV}. Let $V_i:=\calX^{i+1}\setminus F_i$ for $0\leq i\leq m$. Then by definition $F_{i+1}\xrightarrow{\pi_{i+1}}\calX^{i+1}$ is an at most curve over any point of $V_i$ for any $0\leq i\leq m-1$. 
	
	Since $V_m$ surjects to $S$ by the assumption above this lemma,  $F_i$ is a proper subset of $\calX^{i+1}$ for any $0\leq i\leq m$, even fiber-wise over any $s\in S$. In particular, $F_0\subseteq \calX$ is a family of (at most) curves over $S$. 
	
	Let $\calC_0\rightarrow S_0$ be $F_0\rightarrow S$.	Let $\calC_i\rightarrow S_i$ be $F_i\cap \pi_{i}^{-1}V_{i-1}\rightarrow V_{i-1}$ for $1\leq i\leq m$. Let $P\in \calX_s(\bar\QQ)$ for some $s\in S(\bar\QQ)$.
	
	If $P\notin(\calC_0)_s(\bar\QQ)$, then let $j(P)\geq 1$ be the smallest integer $j$ such that the set of points $(P,Q_1,\dots,Q_j)\in\calX^{j+1}_s(\bar\QQ)$ with $Q_i$ close to $P$ for any $1\leq i\leq j$, is contained in $F_j(\bar\QQ)$. Then $j(P)\leq m$. By definition of $j(P)$, there is a point $(P,Q_1,\dots,Q_{j(P)-1})$ with $Q_i$ close to $P$ for any $1\leq i\leq j(P)-1$ in the complement of $F_{j(P)-1}$, namely $V_{j(P)-1}$. Note that in the case $j(P)=1$, the above should be interpreted as: the set $\{Q_1,\dots,Q_{j(P)-1}\}=\emptyset$ and $P$ is in the complement of $F_0$, namely $V_0$.  
	
	By definition, any prolongation $(P,Q_1,\dots,Q_{j(P)-1},Q)$ with $Q$ close to $P$ is in $F_{j(P)}$. So the fiber of $\pi_{j(P)}:F_{j(P)}\rightarrow \calX^{j(P)}$ over the point $(P,Q_1,\dots,Q_{j(P)-1})$ contains all points that are close to $P$. The fiber of $F_{j(P)}$ over any point of $V_{j(P)-1}$, in particular $(P,Q_1,\dots,Q_{j(P)-1})$, is (at most) a curve in the family $\calC_{j(P)}\rightarrow S_{j(P)}$. So all points that are close to $P$ must live in this (at most) curve.
	\end{proof}
	
	Since all $\calX_s$ are GeM, all curves appearing in the families are GeM.  By applying Theorem \ref{curve} to these families of curves, there exists $c_3>0$ such that for any $0\leq i\leq m$, for any $s\in S(\bar\QQ)$ and a subgroup $\Gamma\subseteq \calA_s(\bar\QQ)$ of finite rank $r$, and for any $s'\in S_i(\bar\QQ)$ over $s$, we have $\#(\Gamma\cap (\calC_i)_{s'}(\bar\QQ))\leq c_3^{1+r}$. Let us remark that the degenerate case is trivial. Indeed, we can easily deduce a version of Theorem \ref{curve} with curves replaced by at most curves.
	
	Now if we take $s\in S(\bar\QQ)$ such that $h(\iota_S(s))>\max\{1,\frac{2c_2}{c_1}\}$, then $\frac{c_1}{2}h(\iota_S(s))<c_1h(\iota_S(s))-c_2$. This means, if two points $P,Q$ in $\calX_s(\bar\QQ)$ satisfy $\hat h(P-Q)\leq\frac{c_1}{2}h(\iota_S(s))$, then they are close. Consider a subgroup $\Gamma\subseteq \calA_s(\bar\QQ)$ of finite rank $r$. By induction hypothesis, we may assume $\iota_S:S\rightarrow \bA_g$ is the restriction of a finite morphism $\bar S\rightarrow \bar\bA_g$ for some compactification $\bar S$ of $S$; see Zariski's main theorem \cite[Corollaire 18.12.13]{EGAIVd}. Since the pullback along a finite morphism of an ample line bundle is ample, the height function on $\bar\bA_g$ fixed in \S\ref{section_notation_abelianschemes} induces a height function on $S$ by $\iota_S$ by composition. Note that although the induced height on $S$ is not necessarily from an immersion, a positive constant multiple is and it does not affect the result. So we can still apply Proposition \ref{large_points_dim_2} to see that there exists a constant $c_4>0$ such that the number of large points $P\in\calX_s(\bar\QQ)\cap\Gamma$ with $\hat h(P)>c_4\max\{1,h(\iota_S(s))\}=c_4h(\iota_S(s))$ is at most $c_4^{1+r}$. So it suffices to bound the number of points $P\in\calX_s(\bar\QQ)\cap\Gamma$ with $\hat h(P)\leq c_4h(\iota_S(s))$.
	
	To do this, we use a classical argument by regarding $\Gamma\otimes\RR$ as a normed vector space of dimension $r$ with norm given by $|\cdot|=\hat h^{1/2}$. Let $R_1:=\sqrt{c_4h(\iota_S(s))}$ and $R_2:=\sqrt{\frac{c_1}{2}h(\iota_S(s))}/2$. The closed ball of radius $R_1$ centered at $0$ is covered by at most $(1+2R_1/R_2)^r=(1+4\sqrt{2c_4/c_1})^r$ closed balls of radius $R_2$ by \cite[Lemme 6.1]{remond-decompte}. In any closed ball of radius $R_2$, we have two possibilities. One is that it contains only points in $(\calC_0)_s(\bar\QQ)\cap\Gamma$, which is fine since we know their number is bounded by $c_3^{1+r}$. The other is the opposite: it contains some point $P$ outside $(\calC_0)_s(\bar\QQ)\cap\Gamma$. But by our choice of $R_2$, all the points in the ball are close to $P$. So by Lemma \ref{lemma_surface}, they must belong to a curve $(\calC_i)_{s'}$ for some $1\leq i\leq m$ and $s'\in S_i(\bar\QQ)$ over $s\in S(\bar\QQ)$. Since $\#\left((\calC_i)_{s'}(\bar\QQ)\cap \Gamma\right)\leq c_3^{1+r}$, the number of points in the closed ball of radius $R_2$ is at most $c_3^{1+r}$. In any case, a closed ball of radius $R_2$ contains at most $c_3^{1+r}$ points. So the ball of radius $R_1$ contains at most $c_3^{1+r}\cdot (1+4\sqrt{2c_4/c_1})^r$ points. 
	
	Overall, we have
	\[
	\#(\calX_s(\bar\QQ)\cap\Gamma)\leq c_4^{1+r}+c_3^{1+r}\cdot (1+4\sqrt{2c_4/c_1})^r \leq c^{1+r}
	\]for  $c= \max\{1,2c_2/c_1,c_4+c_3(1+4\sqrt{2c_4/c_1})\}$ whenever $h(\iota_S(s))>c\geq \max\{1,2c_2/c_1\}$. Note that $c$ does not depend on $s\in S(\bar\QQ)$ or $\Gamma$.
\end{proof}

\section{Application}
Consider a smooth irreducible projective curve $C$ of genus $g\geq3$ over a number field $F$. Let $\Jac(C)$ be its Jacobian variety. Suppose there is a quadratic point $P\in C(F')$ for some quadratic extension $F'$ of $F$, and let $P'$ be its Galois conjugate.  The symmetric product $C^{(2)}:=\frac{C\times C}{S_2}$ corresponds to divisors of degree $2$ on $C$. It is well-known that the morphism
\[
\begin{split}
	C^{(2)}&\longrightarrow \Jac(C)\\
	\{Q_1,Q_2\}&\longmapsto [Q_1+Q_2-P-P']
\end{split}
\]
is injective on closed points if $C$ is not hyperelliptic. Let $W_2(C)$ be its scheme-theoretic image. It is shown in \cite[Theorem 2]{SilvermanHarris} that $W_2(C)$ contains no curves of genus $1$ if $C$ is neither hyperelliptic nor bielliptic. Clearly, $W_2(C)$ is not a translate of an abelian surface, since $C$ (hence $W_2(C)$) generates $J(C)$. Namely, $W_2(C)$ is GeM. Therefore Faltings' Theorem \cite{Faltings} tells us that there are only finitely many rational points on $W_2(C)$ which implies the same for $C^{(2)}$. Note that there is a map from the set of quadratic points $C(F,2)$ of $C$ to the set of rational points of $C^{(2)}$, whose fibers consist of at most two points; see the proof of \cite[Corollary 3]{SilvermanHarris} for details. Therefore $\#C(F,2)$ is finite in this case.

Now consider the moduli space $\bM_g$ of smooth irreducible projective curves of genus $g$ over $\bar\QQ$ with level-$l$-structure. It is representable by an irreducible quasi-projective variety over $\bar\QQ$, see for example \cite[Theorem 1.8]{OortSteenbrink}. It is well-known that the locus of hyperelliptic curves of genus $g$ forms a $(2g-1)$-dimensional closed subvariety $\bH_g$ in $\bM_g$; see \cite[XIII.8]{ACG2}. The locus of bielliptic curves of genus $g$ forms a $(2g-2)$-dimensional closed subvariety $\bM_g^{\mathrm{be}}$ in $\bM_g$; see \cite{biellipticlocus}. We are interested in the complement of these two loci and define $\bM_g^{\circ}:=\bM_g\setminus(\bH_g\cup\bM_g^{\mathrm{be}})$, which is an open subset of dimension $3g-3$ in $\bM_g$. 

Let $\gothC_g^\circ\rightarrow \bM_g^\circ$ be the universal non-hyperelliptic non-bielliptic curve of genus $g$ with level-$l$-structure. Let $\calJ_g^\circ:=\Jac(\gothC_g^\circ/\bM_g^\circ)$ be the relative Jacobian. It is a principally polarized abelian scheme with an induced level-$l$-structure. To embed the curves in their Jacobians, we need a quasi-section. By \cite[Corollaire 17.16.2 and 17.16.3]{EGAIVd}, there exists an \'etale surjective quasi-finite morphism $S\rightarrow \bM_g^\circ$ for some affine variety $S$, such that the base change $\gothC:=\gothC_g^\circ\times_{\bM_g^\circ}S\rightarrow S$ has a section $i:S\rightarrow\gothC.$
Write $\calJ\rightarrow S$ for the base change of $\calJ_g^\circ\rightarrow\bM_g^\circ$.
Then we embed $\calC\rightarrow S$ into $\calJ\rightarrow S$ using the section $i$, by sending $\calC_s$ to $\calC_s-i(s)\subseteq \calJ_s$ for any $s\in S$. 

The symmetric product $\calC^{(2)}:=\frac{\calC\times_S \calC}{S_2}$ of $\calC\rightarrow S$ exists since $S$ is quasi-projective. Let $\calX:=W_2(\calC)$ be the image of $\calC^{(2)}$ inside $\calJ$. Let $\iota_S:S\rightarrow \bA_g$ be the modular map induced by $\calJ\rightarrow S$. Then $\iota_S$ is quasi-finite since the Torelli morphism $\bM_g\rightarrow \bA_g$ is quasi-finite and $S\rightarrow \bM_g$ is quasi-finite.

The proof of Theorem \ref{quadratic_points} is similar to \cite[Proof of Theorem 1.1]{DGH21}.

\begin{proof}[Proof of Theorem {\ref{quadratic_points}}]
	Let us fix $l=3$ for this proof and consider level-$3$-structure only.
	
	Take a non-hyperelliptic non-bielliptic smooth geometrically irreducible projective curve $C$ of genus $g$ defined over some number field $F$ with $[F:\QQ]\leq d$. Let $s_0\in S(\bar\QQ)$ be a point over which the fiber $\calX_{s_0}$ corresponding to an embedding of $C^{(2)}$ in the Jacobian (over $\bar\QQ$). Let $s_1:=\iota_S(s_0)\in \bA_g(\bar\QQ)$. The Jacobian of $C$ corresponds to an $F$-rational point $s_2$ of $\bA_{g,1}$, the coarse moduli space of principally polarized abelian varieties of dimension $g$ without level structure.  The fine moduli space of principally polarized abelian varieties of dimension $g$ with level-$3$-structure $\bA_g$ is a quasi-finite cover of $\bA_{g,1}$ with $\bA_g\rightarrow \bA_{g,1}$ defined over $\QQ(\zeta_3)$.  Let $m$ be the maximal geometric cardinality of a fiber of  $\bA_g\rightarrow \bA_{g,1}$. Then $[\kappa(s_1):F]\leq 2m$, whence $[\kappa(s_1):\QQ]\leq 2md$.
	
	By applying Theorem \ref{mainthm} to the family $\calX\rightarrow S$ inside $\calJ\rightarrow S$, there exists a constant $c>1$ such that for any $s\in S(\bar\QQ)$ with $h(\iota_S(s))>c$, and any subgroup $\Gamma\subseteq \calJ_s(\bar\QQ)$ of finite rank $r$, we have $\#(\calX_s(\bar\QQ)\cap\Gamma)\leq c^{1+r}$. The constant $c$ is independent of $s\in S(\bar\QQ)$ and $\Gamma$.

	On the other hand, by Northcott's property, there are only finitely many points defined over a number field of bounded degree $md$ in $\bA_g$ with height at most $c$. Say they correspond to principally polarized abelian varieties $A_1,\dots,A_N$. For this finite set, we apply R\'emond's estimate \cite[page 643]{DP02} to each individual abelian variety and the subvariety. For any $i$ and any GeM subvariety $X\subseteq A_i$, we have $\#(X(\bar\QQ)\cap\Gamma)\leq C_i^{1+r}$ for any subgroup $\Gamma\subseteq A_i(\bar\QQ)$ of rank $r$, where $C_i$ depends on the principally polarized abelian variety $(A_i,L_i)$ and $X$, but not on $\Gamma$. We are only interested in $\calX_s$ over some $A_i$. Hence we can find a constant which by abuse of notation we still denote as $c>1$, such that for any $s\in S(\bar\QQ)$ over some $A_i$, we have $\#(\calX_s(\bar\QQ)\cap\Gamma)\leq c^{1+r}$ for any subgroup $\Gamma\subseteq A_i(\bar\QQ)=\calJ_{s}(\bar\QQ)$ of finite rank $r$.

In short, we get uniformity for the fibers over those $s\in S(\bar\QQ)$ with large height or small degree of its defining field.
		
	To prove the theorem, we may assume there is a quadratic point $P\in C(F')$ for some number field $F'\supseteq F$ with $[F':F]\leq 2$ and let $P'$ denote its Galois conjugate. There is an $F$-morphism from $C^{(2)}$ to $\Jac(C)$ by sending $(Q_1,Q_2)$ to $[Q_1+Q_2-P-P']$. Denote the image by $W_2(C)$. By the discussion in the beginning of this section, $C^{(2)}\rightarrow W_2(C)$ is injective on points, due to the non-hyperelliptic assumption on $C$.  Note that $\Jac(C)=\calJ_{s_0}$. We regard $\Gamma=\Jac(C)(F)$ as a subgroup of $\calJ_{s_0}(\bar\QQ)$. By the Mordell--Weil Theorem, $\Gamma$ is of finite rank. Note also that $W_2(C)$ is a translate of $\calX_s$ by some point $a\in \calA_s(\bar\QQ)$. By our discussion, then
	\[
	\# W_2(C)(F)=\#(W_2(C)(\bar\QQ)\cap\Gamma)\leq\#(\calX_s(\bar\QQ)\cap\langle\Gamma, a\rangle)\leq c^{2+\rho}.
	\]
	
	Note that there is an (at most) two-to-one (not necessarily surjective) map  of sets
	\[
	\begin{split}
		C(F,2)&\rightarrow C^{(2)}(F)	\\
		P&\mapsto \{P,P'\}
	\end{split}
	\]
	where $P'$ is either the conjugate of $P$, or equal to $P$ if $P\in C(F)$. Therefore we have
	\[
	\#C(F,2)\leq 2\cdot \# W_2(C)(F)\leq 2c^{2+\rho}\leq (2c^2)^{1+\rho}.
	\]

\end{proof}

\appendix
\section{Proof of Theorem {\ref{Kuhne_Bogomolov}}}\label{Appendix B}

In this appendix, we reproduce the proof of Theorem \ref{Kuhne_Bogomolov} using the non-degenerate subvariety constructed in Proposition \ref{nondegeneracy}. The argument is a modification of the proof in Gao's survey \cite[Proposition 8.3]{Gao_survey}. We need the following Theorem (see \cite[Lemma 23]{Kuhne} and in particular \cite[Corollary 8.2]{Gao_survey} for the version below, which is a consequence of K\"uhne's equidistribution result \cite[Theorem 1]{Kuhne}. Fix $\bar\QQ\hookrightarrow \CC$.

\begin{theorem}\label{theorem_A}
	Let $S$ be a smooth quasi-projective variety over $\bar\QQ$. Let $\calA\rightarrow S$ be a principally polarized abelian scheme with level-$l$-structure. Let $\calX\subseteq \calA$ be a non-degenerate integral subvariety. Assume $\calX,\calA,S$ and the morphisms between them are defined over a number field $F$. Let $\mu$ be the equilibrium measure on $\calX(\CC)$, which is a constant multiple of $\omega^{\wedge \dim \calX}$ with $\omega$ the Betti $(1,1)$-form on $\calX$; see \cite[Proposition 2.2]{DGH21}. Then for any continuous function $f$ on $\calX(\CC)$ with compact support and $\epsilon>0$, there exists $\delta=\delta(f,\epsilon)>0$, such that the union
	\[
	\{x\in\calX(\bar\QQ):\hat h(x)\geq \delta\}\cup 
	\left \{x\in \calX(\bar\QQ): \left|\frac{1}{\#O(x)}\sum_{x'\in O(x)}f(x')- \int_{\calX(\CC)}f\mu\right|< \epsilon \right\}
	\]
	contains a dense open subset of $\calX$, where $O(x)$ denotes the Galois orbit of $x$ over $F$.
\end{theorem}

Now let us prove Theorem \ref{Kuhne_Bogomolov}.

Assume $\calC,\calA,S$ and the morphisms between them are defined over a number field $F$.

Induct on $\dim S$. By the induction hypothesis, we may assume without loss of generality that $S$ is smooth quasi-projective, $\calC\rightarrow S$ is flat and $\calC^{\mathrm{sm}}\rightarrow S$ is smooth. These assumptions will be convenient later. For the family of curves $\calC\rightarrow S$, there is an induced morphism from $S$ to the Hilbert scheme $\Hilb(\gothA_g/\bA_g)$. If this morphism is not generically finite, then the scheme-theoretic image, say $S'$, is a variety of lower dimension. Denote the family induced by $S'$ as $\calC'\rightarrow S'$ and the abelian scheme by $\calA'\rightarrow S'$. By induction, the result holds for the family $\calC'\rightarrow S'$. Since $\calC\rightarrow S$ is just the pullback of $\calC'\rightarrow S'$, we get the result for $\calC\rightarrow S$ naturally.

Thus we may as well assume $S\rightarrow \Hilb(\gothA_g/\bA_g)$ is generically finite. Then Proposition \ref{nondegeneracy} applies (the geometric generic fiber $\calC_{\bar\eta}$ is integral and generates $\calA_{\bar\eta}$; see the earlier short version of the proof of Theorem \ref{Kuhne_Bogomolov}) and there exists $n$ such that $\calC^n\subseteq\calA^n$ is non-degenerate. It is this non-degenerate subvariety that we will study. Write $\calX:=\calC^n$. Then $\Stab(\calX_{\bar\eta})$ is finite.


Consider the following variant of the restricted Faltings--Zhang morphism 
\[
\begin{split}
	\calD:\calX^{m+1}&\rightarrow \calA^{nm}\\
	(\bx_0,\dots,\bx_m) &\mapsto(\bx_1-\bx_0,\dots,\bx_m-\bx_{m-1}),
\end{split}
\]which is clearly defined over $F$. Since $\Stab(\calX_{\bar\eta})$ is finite, for $m$ large enough this map is of generic degree equal to $d:=\#\Stab(\calX_{\bar\eta})$. One way of seeing this is suggested to us by the referee as follows. Clearly the fiber over $\calD(\bx_0,\dots,\bx_m)$ is
$\bigcap_{i=0}^m (\calX-\bx_i).$ An easy inductive argument shows that the generic sequence of $\bx_0,\dots,\bx_m$ for any $m\geq \dim (\calX_{\bar\eta})$ will cut down the dimension of the fiber to zero; see \cite[Lemme 2.1]{remond-inegalite} for details. Note that this is essentially because the containment of a closed subvariety is a closed condition. By intersecting once more using a generic point, we can always cut down the cardinality of the finite set to $\#\Stab(\calX_{\bar\eta})$ in $m=\dim (\calX_{\bar \eta})+1$ steps. Fix this $m$.

The image $\image\calD$ is not necessarily non-degenerate, but $\calX\times_S\image\calD\subseteq \calA^{n(m+1)}$ is non-degenerate (an easy exercise since a Betti map for the product abelian scheme is the product of two Betti maps for two abelian schemes). Write $\calY:=\calX\times_S\image\calD$. Write $\calD':=\id_{\calX}\times_S\calD:\calX^{m+2}\rightarrow\calY$. Note that $\calX^{m+2}\subseteq \calA^{n(m+2)}$ and $\calY\subseteq\calA^{n(m+1)}$ are irreducible and non-degenerate. We will compare the equilibrium measures $\mu_1$ on $\calX^{m+2}(\CC)$ and $\mu_2$ on $\calY(\CC)$. 

Claim: $\mu_1\neq (\deg\calD')^{-1}\calD'^*\mu_2$ near the point $(\bx_0,\bx_0,\dots,\bx_0)$, where $\bx_0$ is a fixed smooth algebraic point of $\calX$ such that the Betti rank at $\bx_0$ is maximal, whence the Betti form at $\bx_0$ is positive by \cite[Lemma 11]{Kuhne}. Note that here the pullback $\calD'^*\mu_2$ is defined through the pullback of the form defining the measure. Note also that $(\bx_0,\bx_0,\dots,\bx_0)$ is a smooth point of $\calX^{m+2}(\CC)$.

In fact, since $\calD'(\bx_0,\bx_0,\dots,\bx_0)=(\bx_0,0,\dots,0)$ and the fiber of $\calD'$ over $(\bx_0,0,\dots,0)$, containing all points of the form $(\bx_0,\bx,\dots,\bx)$, is of nonzero dimension, the differential of $\calD'$ has nonzero kernel. But $\mu_1\neq 0$ at $(\bx_0,\bx_0,\dots,\bx_0)$ since that the Betti form at $\bx_0$ is positive implies the Betti form at $(\bx_0,\dots,\bx_0)$ is positive; see for instance \cite[Lemma 25]{Kuhne}.

By the claim, there exists a continuous function $f'$  on $\calX^{m+2}(\CC)$ with compact support,  and some $\epsilon>0$ such that 
\begin{equation}\label{eqn_measure}
    \left| \int_{\calX^{m+2}(\CC)}f'\mu_1 -(\deg\calD')^{-1}\int_{\calX^{m+2}(\CC)} f'\calD'^*\mu_2\right|\geq 2\epsilon.
\end{equation}
We can actually choose $f'$ in a way such that $f'$ comes from a compactly supported function $f$ on $\calY(\CC)$. Indeed, spread out $\Stab(\calX_{\bar\eta})$ to a finite \'etale group scheme $G$ over $S$. Shrinking $S$ if necessary, assume moreover that $G(S)$ is naturally isomorphic to the generic fiber. Assume that $\Stab(\calX_s)=G_s$ for $s\in S(\CC)$. There is a \emph{Zariski dense} open subset $V\subseteq \calY$ such that $\calD'$ restricted to $U=\calD'^{-1}V\subseteq  \calX^{m+2}$ is a finite morphism such that the fibers are $G_s$-torsors for any point of $V$ over $s\in S(\CC)$. A standard analysis near $(\bx_0,\dots,\bx_0)$ shows that there exists $f'$ supported on a compact subset of $U$, such that \eqref{eqn_measure} holds. Simply take $f'^*(x):=\sum_{\sigma\in G(S)} f'(x+ \sigma(s))$ for any $x\in \calX^{m+2}_s(\CC)$. Also note that replacing $f'$ by $f'^*$ multiplies the left hand side of \ref{eqn_measure} by $\deg\calD'$ since $\mu_1,\calD'^*\mu_2$ are both invariant under ``translation by G(S)''. Then we can descend $f'^*$ to $f$ on $V$ since $f'^*$ is $G$-invariant. 

We apply Theorem \ref{theorem_A} twice, for the triples $(\calX^{m+2},f',\epsilon)$ and $(\calY,f,\epsilon)$ respectively. There is $\delta=\delta(f,\epsilon)>0$ such that for any $\bx\in U(\bar\QQ)$ and $\by\in V(\bar\QQ)$, we have
\[
\hat h(\bx)\geq \delta\quad\text{or}\quad  \left|\frac{1}{\#O(\bx)}\sum_{\bx'\in O(\bx)}f'(\bx')- \int_{\calX^{m+2}(\CC)}f'\mu_1\right|\geq \epsilon
\]and
\[
\hat h(\by)\geq \delta\quad\text{or}\quad \left|\frac{1}{\#O(\by)}\sum_{\by'\in O(\by)}f(\by')- \int_{\calY(\CC)}f\mu_2\right|\geq \epsilon.
\]
For an algebraic  point $\by\in V(\bar\QQ)$ and any algebraic point $\bx\in U(\bar\QQ)$ such that $\calD'(\bx)=\by$, it is clear that
\[
\frac{1}{\#O(\bx)}\sum_{\bx'\in O(\bx)}f'(\bx')=\frac{1}{\#O(\by)}\sum_{\by'\in O(\by)}f(\by'),
\]using the fact that $\calD'|_V$ is a finite morphism defined over $F$.

If we have
\[
 \left|\frac{1}{\#O(\bx)}\sum_{\bx'\in O(\bx)}f'(\bx')- \int_{\calX^{m+2}(\CC)}f'\mu_1\right|< \epsilon\quad\text{and}\quad \left|\frac{1}{\#O(\by)}\sum_{\by'\in O(\by)}f(\by')- \int_{\calY(\CC)}f\mu_2\right|< \epsilon,
\]
then 
\[
\left| \int_{\calX^{m+2}(\CC)}f'\mu_1 -\int_{\calY(\CC)} f\mu_2\right|< 2\epsilon,
\]
contradicting the choice of $f'$. So we must have either $\hat h(\bx)\geq\delta$ or $\hat h(\by)\geq\delta$. Assume $\bx=(\bz_0,\bz_1,\dots,\bz_{m+1})$. Then $\by=\calD'\bx=(\bz_0,\bz_2-\bz_1,\dots,\bz_{m+1}-\bz_m)$. So 
\[
\hat h(\by)\leq \hat h(\bz_0)+2[\hat h(\bz_2)+\hat h(\bz_1)]+\dots+2[\hat h(\bz_{m+1})+\hat h(\bz_m)]\leq 4\hat h(\bx).
\]
Therefore we always have $\hat h(\bx)\geq \frac{1}{4}\delta$.

Overall, we found a dense open subset $U$ of the fiber product $\calX^{m+2}=\calC^{n(m+2)}$ such that for any $\bx\in U(\bar\QQ)$, we have $\hat h(\bx)\geq \frac{1}{4}\delta$. Using the induction hypothesis, we may assume the projection $U\rightarrow S$ is surjective. Take $c_1=\frac{1}{4n(m+2)}\delta$. We claim that there is $c_2>0$ such that for any $s\in S(\bar\QQ)$, the number of small points
\[
\#\{P\in\calC_{s}(\bar\QQ):\hat h(P)<c_1\}\leq c_2.
\] 
This is an application of \cite[Lemma 7.3]{Gao_survey}. In fact, by the lemma, there is $c_2>0$, independent of $s\in S(\bar\QQ)$, such that if $\Sigma\subseteq\calC_s(\bar\QQ)$ and $\#\Sigma> c_2$, then $\Sigma^{n(m+2)}\cap U_s(\bar\QQ)\neq \emptyset$. But obviously 
\[
\{P\in\calC_s(\bar\QQ):\hat h(P)<c_1\}^{n(m+2)}\cap U_s(\bar\QQ)=\emptyset,
\]
by our choice of $c_1$. So we are done!

\section*{Acknowledgement}
I am extremely grateful to my advisor Dan Abramovich for introducing me to this interesting topic and the arithmetic world. His guidance and insights, both mathematically and mentally, help me survive the hard quarantine days. I am deeply indebted to Ziyang Gao for the enlightening discussions during the preparation of this paper and the insightful comments to the first draft. I would also like to thank Joseph Silverman for his wonderful lectures which triggered my interest in arithmetic geometry. 

I would like to thank the referee wholeheartedly for the thorough reading, insightful comments and selfless help. Without the referee's patience and tolerance, this paper would have been pointless.

\bibliographystyle{plain}
\bibliography{Arithmeticgeometry}
\end{document}